\documentclass[12pt]{article}
\usepackage{amsmath, amsthm, amsfonts, amssymb}
\usepackage{enumerate}
\usepackage[all,cmtip]{xy}
\usepackage[margin=2.5cm]{geometry}
\usepackage{parskip}

\numberwithin{equation}{section}

\theoremstyle{plain}
\newtheorem{thm}{Theorem}[section]
\newtheorem{lemma}[thm]{Lemma}
\newtheorem{prop}[thm]{Proposition}

\theoremstyle{definition}
\newtheorem{defn}{Definition}[section]

\newtheorem{rmrk}[thm]{Remark}

\DeclareMathOperator{\ch}{Ch}
\DeclareMathOperator{\sch}{SCh}
\DeclareMathOperator{\disc}{disc}
\DeclareMathOperator{\en}{End}
\DeclareMathOperator{\Hom}{Hom}

\DeclareMathOperator{\LHS}{LHS}
\DeclareMathOperator{\RHS}{RHS}
\DeclareMathOperator{\res}{Res}
\DeclareMathOperator{\spa}{Span}
\DeclareMathOperator{\tr}{Tr}
\DeclareMathOperator{\str}{STr}
\DeclareMathOperator{\zhu}{Zhu}

\newcommand{\al}{\alpha}
\newcommand{\G}{\Gamma}
\newcommand{\ga}{\gamma}
\newcommand{\D}{\Delta}
\newcommand{\eps}{\epsilon}
\newcommand{\la}{\lambda}
\newcommand{\om}{\omega}
\newcommand{\vp}{\varphi}

\newcommand{\C}{\mathbb{C}}
\newcommand{\R}{\mathbb{R}}
\newcommand{\Q}{\mathbb{Q}}
\newcommand{\Z}{\mathbb{Z}}

\newcommand{\g}{\mathfrak{g}}
\newcommand{\h}{\mathfrak{h}}

\newcommand{\CC}{\mathcal{C}}
\newcommand{\HH}{\mathcal{H}}
\newcommand{\M}{\mathcal{M}}
\newcommand{\OO}{\mathcal{O}}
\newcommand{\V}{\mathcal{V}}
\newcommand{\W}{\mathcal{W}}

\newcommand{\slmat}
{\left(\begin{smallmatrix}{\sf a} & {\sf b} \\ {\sf c} & {\sf d} \end{smallmatrix}\right)}
\newcommand{\slmatbig}{\left(\begin{array}{cc}{\sf a}&{\sf b}\\{\sf c}&{\sf d}\\ \end{array}\right)}
\newcommand{\ov}{\overline}
\newcommand{\oq}{\overline{q}}
\newcommand{\vac}{\left|0\right>}
\newcommand{\DD}{\nabla}

\newcommand{\ide}{X}
\newcommand{\suv}{\xi}

\title{Modular Invariance for Twisted Modules over a Vertex Operator Superalgebra}
\date{}
\author{Jethro Van Ekeren\footnote{email: jethrovanekeren@gmail.com}\\ \small{\emph{Department of Mathematics, MIT, Cambridge, Massachusetts 02139, USA}}}

\begin{document}

\maketitle

\begin{abstract}
\noindent The purpose of this paper is to generalize Zhu's theorem about characters of modules over a vertex operator algebra graded by integer conformal weights, to the setting of a vertex operator superalgebra graded by rational conformal weights. To recover $SL_2(\Z)$-invariance of the characters it turns out to be necessary to consider twisted modules alongside ordinary ones. It also turns out to be necessary, in describing the space of conformal blocks in the supersymmetric case, to include certain `odd traces' on modules alongside traces and supertraces. We prove that the set of supertrace functions, thus supplemented, spans a finite dimensional $SL_2(\Z)$-invariant space. We close the paper with several examples.
\end{abstract}

\section{Introduction}

%%%%%%%%%%%%%%%%%%%%%%%%%%%%%%%%%%%%%%%%%%%%%%%%%%%%%%%%%%%%%%%%%%%%%%%%%%%%%%%%%%%%%%%%
%%%%%%%%%%%%%%%%%%%%%%%%%%%%%%%%%%%%%%%%%%%%%%%%%%%%%%%%%%%%%%%%%%%%%%%%%%%%%%%%%%%%%%%%
%%%%%%%%%%%%%%%%%%%%%%%%%%%%%%%%%%%%%%%%%%%%%%%%%%%%%%%%%%%%%%%%%%%%%%%%%%%%%%%%%%%%%%%%
%%%%%%%%%%%%%%%%%%%%%%%%%%%%%%%%%%%%%%%%%%%%%%%%%%%%%%%%%%%%%%%%%%%%%%%%%%%%%%%%%%%%%%%%
%%%%%%%%%%%%%%%%%%%%%%%%%%%%%%%%%%%%%%%%%%%%%%%%%%%%%%%%%%%%%%%%%%%%%%%%%%%%%%%%%%%%%%%%
%%%%%%%%%%%%%%%%%%%%%%%%%%%%%%%%%%%%%%%%%%%%%%%%%%%%%%%%%%%%%%%%%%%%%%%%%%%%%%%%%%%%%%%%

Let $\g$ be a finite dimensional simple Lie algebra and let $\hat{\g}$ be the corresponding affine Kac-Moody algebra. In \cite{KacPeter} Kac and Peterson expressed the (normalized) characters of the integrable $\hat{\g}$-modules in terms of Jacobi theta functions. In particular they showed that the normalized characters of the integrable modules at a fixed level $k \in \Z_+$ span an $SL_2(\Z)$-invariant vector space. Later the monstrous moonshine conjecture, relating the monster finite simple group to the modular $j$-function, was resolved by Borcherds \cite{Bor}, \cite{Borlater} using generalized Kac-Moody algebras and the monster vertex operator algebra of Frenkel, Lepowsky and Meurman \cite{FLM}. Later still Zhu \cite{Zhu} established $SL_2(\Z)$-invariance of the characters of an arbitrary $C_2$-cofinite rational vertex operator algebra with integer conformal weights, see Theorem \ref{Zhuthm} below. The modular invariance of the two prior examples can be recovered as special cases of Zhu's result. The main result of this paper, Theorem \ref{mythm} below, is a generalization of Theorem \ref{Zhuthm}.

Recall that a vertex operator algebra (VOA) consists of a vector space $V$, two distinguished vectors $\vac$ and $\om$ (called the vacuum and Virasoro vector, respectively), and an assignment to each vector $u \in V$ of a `field' $Y(u, z) = \sum_{n \in \Z} u_{(n)} z^{-n-1}$ where $u_{(n)} \in \en V$. These data are to satisfy certain axioms (see Definition \ref{VOSAdefn}).

By definition $V$ is $C_2$-cofinite if $V_{(-2)}V$ has finite codimension in $V$.

If $Y(\om, z) = L(z) = \sum_{n \in \Z} L_n z^{-n-2}$ then the operators $L_n$ form a representation on $V$ of the Virasoro algebra with some central charge $\mathfrak{c}$. The energy operator $L_0$ acts semisimply. The eigenvalue of an $L_0$-eigenvector $u \in V$ is called its conformal weight and is denoted $\D_u$. We write $V_k$ for the set of vectors of conformal weight $k$.% We put $u_n = u_{(n + \D_u - 1)}$.

A module over a VOA $V$ is a vector space $M$ together with a field
\begin{align} \label{indexexpl}
Y^M(u, z) = \sum_{n \in \Z} u^M_{(n)} z^{-n-1} = \sum_{n \in -\D_u + \Z} u^M_{n} z^{-n-\D_u}
\end{align}
assigned to each $u \in V$. These data are to satisfy certain axioms, see Definition \ref{twisteddefn}. A positive energy $V$-module is an $\R_+$-graded $V$-module $M = \oplus_j M_j$, with finite dimensional graded pieces, such that $u^M_n M_j \subseteq M_{j-n}$ for each $u \in V$, $n \in -\D_u + \Z$. We say that $V$ is rational if it has finitely many irreducible positive energy modules and every positive energy $V$-module is a direct sum of irreducible ones.

Given $V$ as above, Zhu introduced a second VOA structure $Y[u, z] = \sum_{n \in \Z} u_{([n])} z^{-n-1}$ on $V$, and a new Virasoro element $\tilde{\om} \in V$ (see Definition \ref{Zhustructure}). Let $L_{[0]}$ be the new energy operator and $V_{[k]}$ the set of vectors of conformal weight $k$ with respect to $\tilde{\om}$.

Let $\HH$ denote the complex upper half plane and set $q = e^{2\pi i \tau}$ where $\tau$ is a variable on $\HH$. Recall the standard weight $k$ action of the modular group $SL_2(\Z)$ on holomorphic functions on $\HH$:
\begin{align} \label{sl2action}
[f \cdot A](\tau) = ({\sf c}\tau+{\sf d})^{-k} f\left(\frac{{\sf a}\tau + {\sf b}}{{\sf c}\tau + {\sf d}}\right) \quad \text{for} \quad A = \slmatbig \in SL_2(\Z).
\end{align}
\vspace{.3cm}
\begin{thm}[Zhu] \label{Zhuthm}
Let $V$ be a $C_2$-cofinite rational VOA with non negative integer conformal weights. Let $k \in \Z_+$, $u \in V_{[k]}$, and let $M$ be an irreducible positive energy $V$-module. Then the trace function
\[
\tr_M u^M_0 q^{L_0 - \mathfrak{c}/24}
\]
converges to a holomorphic function $S_M(u, \tau)$ of $\tau \in \HH$. Let $\CC(u)$ denote the span of $S_M(u, \tau)$ as $M$ runs over the set of irreducible positive energy $V$-modules. Then $\CC(u)$ is a finite dimensional vector space invariant under the weight $k$ action of $SL_2(\Z)$.
\end{thm}
Setting $u = \vac$ shows that the graded characters of the irreducible positive energy $V$-modules span an $SL_2(\Z)$-invariant space of weight $0$. Modular forms of other integer weights arise from traces of other elements of $V$.

Theorem \ref{Zhuthm} has been generalized in several directions. Dong, Li and Mason \cite{DLMorbifold} proved a `twisted' version for a VOA $V$ (again required to be rational, $C_2$-cofinite, and with integer conformal weights) together with a finite group $G$ of its automorphisms. We shall describe their result in some detail below.

Dong and Zhao gave further generalizations to the case of a rational $C_2$-cofinite \emph{vertex operator superalgebra (VOSA)} $V$ (see Definition \ref{VOSAdefn}) together with a finite group $G$ of its automorphisms. In \cite{DZ2} these authors dealt with the case of integer conformal weights, and in \cite{DZ} with the case in which even elements of $V$ have conformal weights in $\Z$ and odd elements have conformal weights in $1/2 + \Z$.

In the VOA setting Miyamoto \cite{M2} established $SL_2(\Z)$-invariance of traces of \emph{intertwining operators} of type $\binom{M^2}{M^1 \quad M^2}$. In this context intertwining operators of type $\binom{M^3}{M^1 \quad M^2}$ are certain maps $I : M^1 \rightarrow \Hom(M^2, M^3)[[z, z^{-1}]]$, where $M^1$, $M^2$ and $M^3$ are $V$-modules, satisfying a natural analog of the $V$-module condition. Theorem \ref{Zhuthm} is the special case $M^1 = V$. Later Yamauchi \cite{Y} proved a twisted version of Miyamoto's result. Generalization to traces of arbitrary products of intertwining operators was later achieved by Huang \cite{Huang}. Although we do not consider intertwining operators in this paper, there is a feature in common with these other works: the appearance of modular forms of non integer weight.

Finally we mention some other works in the spirit of \cite{Zhu}. In \cite{Mnonrational} Theorem \ref{Zhuthm} is generalized to nonrational VOAs \cite{Mnonrational} (see also \cite{Arike} in this connection), and in \cite{MasonTuiteZuevsky} and \cite{Jordan} Zhu's recursion relations for $n$-point functions are generalized to the VOSA case.

%In \cite{Mnonrational} Miyamoto generalized Zhu's theorem to the case of nonrational $C_2$-cofinite VOAs with integer conformal weights (see also \cite{Arike}). We deal in this paper with VOSAs which are not rational as VOAs but nevertheless are rational as VOSAs. A generalization of Zhu's theorem to arbitrary nonrational VOSAs has not yet been carried out.

The present paper is concerned with VOSAs graded by rational (not necessarily integer, nor half-integer) conformal weights. There are many natural examples from which to draw motivation, some of which we review below. We do not assume positivity of conformal weights, we work in the general supersymmetric setting, and we impose no relation between the parity of a vector and its conformal weight. To recover $SL_2(\Z)$-invariance of characters in this setting it turns out to be necessary to consider twisted modules (see Definition \ref{twisteddefn}). It is thus natural to follow \cite{DLMorbifold}, \cite{DZ}, and \cite{DZ2} by working from the outset in the greater generality of a VOSA together with a finite group of its automorphisms, and our main result, Theorem \ref{mythm}, is a generalization of the corresponding main theorems of those papers. Dealing with superalgebras leads naturally to the notion of `odd trace' (see below) which we extend to the VOSA setting. Some complication is introduced into the statement of Theorem \ref{mythm} by the interaction of the odd trace with the twisting.

We now describe the result of \cite{DLMorbifold} in more detail. Let $V$ be a VOA and $G$ a finite group of its automorphisms. For $g \in G$ a $g$-twisted $V$-module is a vector space $M$ together with fields $Y^M(u, z) = \sum_n u^M_{n} z^{-n-\D_u}$ where the sum runs over $n \in \eps_u + \Z$ (instead of $-\D_u + \Z$) where $\eps_u$ is a certain real number depending on $u \in V$ and $g$. There is an obvious notion of `$g$-rational' VOA. Fix commuting elements $g, h \in G$, and let $M$ be a $g$-twisted $V$-module. Setting $Y^{h \cdot M}(u, z) = Y^M(h(u), z)$ equips $M$ with another $g$-twisted $V$-module structure denoted $h \cdot M$. Suppose $M$ is $h$-invariant, meaning that $h \cdot M$ is equivalent to $M$, and let $\ga : M \rightarrow M$ witness the equivalence, i.e.,
\begin{align} \label{gammaprop}
h(u)^M_{n} = \ga^{-1} u^M_{n} \ga \quad \text{for all $u \in V$, $n \in \eps_u + \Z$}.
\end{align}
%\vspace{.3cm}
\begin{thm}[Dong, Li, and Mason] \label{DLMthm}
Let $V$ be a $C_2$-cofinite VOA with integer conformal weights, such that $V_k = 0$ for sufficiently negative $k$. Let $G$ be a finite group of automorphisms of $V$, and suppose $V$ is $g$-rational for each $g \in G$. Let $k \in \Z$, $u \in V_{[k]}$, let $g, h \in G$ commute, and let $M$ be a $h$-invariant irreducible positive energy $g$-twisted $V$-module. Introduce $\ga : M \rightarrow M$ as in equation (\ref{gammaprop}). Then the trace function
\[
\tr_M u^M_0 \ga q^{L_0 - \mathfrak{c}/24}
\]
converges to a holomorphic function $S_M(u, \tau)$ of $\tau \in \HH$. Let $\CC(g, h; u)$ denote the span of $S_M(u, \tau)$ as $M$ runs over the set of $h$-invariant irreducible positive energy $g$-twisted $V$-modules. Then $\CC(g, h; u)$ is a finite dimensional vector space invariant under the weight $k$ action (\ref{sl2action}) of $SL_2(\Z)$ in the sense that
\[
\slmatbig : \CC(g, h; u) \rightarrow \CC(g^{\sf a}h^{\sf c}, g^{\sf b}h^{\sf d}; u).
\]
\end{thm}

We work with superalgebras (see Sections \ref{definitions} and \ref{symmfunc}) and so there is overlap with \cite{DZ} and \cite{DZ2} in this respect. One difference, though, is that in those papers a subspace of a vector superspace is allowed to be non $\Z/2\Z$-graded, while in the present work we prefer to allow only $\Z/2\Z$-graded subspaces.

Let $A$ be a finite dimensional simple superalgebra, and let $f : A \rightarrow \C$ be a supersymmetric function on $A$, i.e., $f(ab) = p(a, b) f(ba)$ for all $a, b \in A$. Then one of the following is true.
\begin{itemize}
\item $A = \en(\C^{m|k})$, and $f$ is a scalar multiple of the map $a \mapsto\str_{N}(a)$ where $N = \C^{m|k}$ is the unique irreducible $A$-module.

\item $A = Q_n = \en(\C^n)[\suv] / (\suv^2 = 1)$ where $\suv$ is an odd indeterminate, and $f$ is a scalar multiple of the map $a \mapsto \tr_{N}(a \suv)$ where $N = \C^n + \suv \C^n$ is the unique irreducible $A$-module.
\end{itemize}
We shall refer to these two cases as {\sf Type I} and {\sf Type II} respectively. The {\sf Type II} superalgebra $Q_n$ is often called the `queer superalgebra', it can be realized inside $\en(\C^{n|n})$ as the set of block matrices of the form $\left(\begin{smallmatrix}A & B \\ B & A \end{smallmatrix}\right)$ where $A, B \in \en(\C^n)$. The supersymmetric function on $Q_n$, which becomes $\left(\begin{smallmatrix}A & B \\ B & A \end{smallmatrix}\right) \mapsto \tr(B)$, is often called the `odd trace'.

Now let $V$ be a VOSA with rational conformal weights, and let $G$ be a finite group of its automorphisms. Fix commuting elements $g, h \in G$. An associative superalgebra $\zhu_g(V)$ was defined by De Sole and Kac in \cite{DK}, generalizing constructions in \cite{DZother}, \cite{Xu}, \cite{DLMzhualgebra}, \cite{KacWang}, and the original \cite{Zhu} (see Section \ref{coeffs}). This Zhu algebra has the property that there is a functorial bijection $L$ from the set of irreducible $\zhu_g(V)$-modules to the set of irreducible positive energy $g$-twisted $V$-modules. The automorphism $h$ descends to $\zhu_g(V)$ and permutes its irreducible modules. The bijection $L$ restricts to $h$-invariant modules.

Let $M$ be a $h$-invariant irreducible positive energy $g$-twisted $V$-module, $N$ the corresponding $\zhu_g(V)$-module, and $A$ the corresponding simple component of $\zhu_g(V)$. Let $\ga : M \rightarrow M$ satisfy equation (\ref{gammaprop}). If $A$ is of {\sf Type I} then $\ga$ is unique up to a scalar multiple (and is pure even or odd), but if $A$ is of {\sf Type II} then $\ga$ can be chosen to be either even or odd. Indeed if $A = A_0[\suv]/(\suv^2=1)$ then it turns out $h(\suv) = \pm \suv$, and we shall choose the parity of $\ga$ to be even (resp. odd) in the case $+$ (resp. $-$). We now associate to $M$ the supertrace function
\begin{align}\label{giveaway}
S_{M}(u, \tau) = \left\{ \begin{array}{ll}
\str_{M} \left( u_0 \ga \sigma_{M}^{p(\ga)} q^{L_0 - \mathfrak{c}/24} \right) & \text{if $M$ is of {\sf Type I}}, \\
\tr_{M} \left( u_0 \ga \sigma_{M}^{p(\ga)} \xi q^{L_0 - \mathfrak{c}/24} \right) & \text{if $M$ is of {\sf Type II}}, \\
\end{array}\right.
\end{align}
where by $\suv : M \rightarrow M$ we mean the lift to $M$ of the corresponding $\suv : N \rightarrow N$. We now have
\vspace{.3cm}
\begin{thm}[Main Theorem] \label{mythm}
Let $V$ be a $C_2$-cofinite VOSA with rational conformal weights, such that $V_k = 0$ for sufficiently negative $k$. Let $G$ be a finite group of automorphisms of $V$, and suppose $V$ is $g$-rational for each $g \in G$. Let $k \in \Q$, $u \in V_{[k]}$, let $g, h \in G$ commute, and let $M$ be a $h$-invariant irreducible positive energy $g$-twisted $V$-module. Introduce $\ga : M \rightarrow M$ and $S_M$ as in (\ref{giveaway}). Then $S_M$ converges to a holomorphic function of $\tau \in \HH$. Let $\CC(g, h; u)$ denote the span of $S_M(u, \tau)$ as $M$ runs over the set of $h$-invariant irreducible positive energy $g$-twisted $V$-modules. Then $\CC(g, h; u)$ is a finite dimensional vector space invariant under the weight $k$ action (\ref{sl2action}) of $SL_2(\Z)$ in the sense that
\begin{align}\label{whereCBgo}
\slmatbig : \CC(g, h; u) \rightarrow \CC(g^{\sf a}h^{\sf c}, g^{\sf b}h^{\sf d}; u).
\end{align}
\end{thm}
Now we make some remarks. The condition of $g$-rationality of $V$ in Theorem \ref{mythm} can be replaced by the condition that $\zhu_g(V)$ be finite dimensional and semisimple, which is implies by $g$-rationality (see \cite{DLMzhualgebra} Theorem 8.1). In fact finite dimensionality of $\zhu_g(V)$ is already assured by $C_2$-cofiniteness of $V$ (see \cite{DK} Proposition 2.17(c)).

The condition of $C_2$-cofiniteness imposed in Theorem \ref{mythm} (and in all works cited above), can be replaced by the weaker condition described in Remark \ref{C2remark}. At the end of the introduction we explain an application of Theorem \ref{mythm} made possible by this observation.

The definition of the $SL_2(\Z)$-action involves the term $\gamma_A(\tau) = ({\sf c}\tau + {\sf d})^{-k}$ where $k \notin \Z$. We resolve the ambiguity by defining this term as a principal value, see Section \ref{cbdefn}. Because of this the equation
\[
\gamma_B(\tau)^{-k} \gamma_{A}(B\tau)^{-k} = \gamma_{AB}(\tau)^{-k}
\]
(which is true for $k \in \Z$) holds up to a multiplicative root of unity factor. Therefore the space $\CC(u) = \oplus_{g, h \in G} \CC(g, h; u)$ is only a projective representation of $SL_2(\Z)$ in general. If $\D_u \in \Z$ then it is a true representation.

Let $V$ be a VOSA of the type considered in \cite{DZ}, viz. $\D_u$ lies in $\Z$ (resp. $1/2 + \Z$) for $u$ even (resp. odd), and let $G = \{1, \sigma_V\} \cong \Z / 2\Z$. If we restrict attention to $u = \vac$ then odd trace functions vanish and we may ignore modules of {\sf Type II}. We then recover the result: the vector space spanned by each of the following two sets of functions is $SL_2(\Z)$-invariant of weight $0$.
\begin{itemize}
\item The supercharacters $\str_M q^{L_0 - \mathfrak{c}/24}$ of the $1$-twisted irreducible positive energy $V$-modules.

\item The characters $\tr_M q^{L_0 - \mathfrak{c}/24}$ \underline{and} supercharacters $\str_M q^{L_0 - \mathfrak{c}/24}$ of the $1$-twisted \underline{and} $\sigma_V$-twisted irreducible positive energy $V$-modules.
\end{itemize}
In the physics literature $1$-twisted modules are often referred to as `Ramond twisted' modules, and $\sigma_V$-twisted modules as `Neveu-Schwartz twisted' modules.

Now we indicate the layout of the paper. Like the other generalizations of Theorem \ref{Zhuthm} cited above, our proof follows the pattern of the original paper of Zhu \cite{Zhu}. We have made some simplifications, on the other hand some complications are forced on us by the more general setting.

In Section \ref{definitions} we give basic definitions pertaining to superspaces, superalgebras, VOSAs and their modules. In Section \ref{modforms} we collect some necessary modular form identities. In Section \ref{cbdefn} we define a certain space $\CC(g, h)$ of maps $V \times \HH \rightarrow \CC$ (linear in $V$, holomorphic in $\HH$) called \emph{conformal blocks}, and we determine how the conformal blocks transform under $SL_2(\Z)$.

The $C_2$-cofiniteness condition on $V$ implies that for $S \in \CC(g, h)$ and fixed $u \in V$, the function $S(u, \tau)$ satisfies a Fuchsian differential equation. Moreover there is a \emph{Frobenius expansion} of $S$ in powers of $q$ and $\log q$ whose coefficients are linear maps $V \rightarrow \C$. We sketch the proofs in Section \ref{ODEsection}, referring to \cite{DLMorbifold} for details.

In Section \ref{coeffs} we analyze the leading coefficients in the Frobenius expansion of a conformal block. These coefficients descend to linear maps $\zhu_g(V) \rightarrow \C$. We establish that these maps are $h$-supersymmetric functions on $\zhu_g(V)$ (see Section \ref{symmfunc} for the definition). In Sections \ref{symmfunc} and \ref{tracefunc} we construct a basis of $h$-supersymmetric functions on $\zhu_g(V)$ and extend each one to a supertrace function on $V$, arriving at the definition in equation (\ref{giveaway}) above.

We then prove that the $S_M(u, \tau)$ lie in $\CC(g, h)$. Finally in Section \ref{exhaust} we prove that the $S_M(u, \tau)$ span $\CC(g, h)$. Theorem \ref{mythm} is obtained by combining this result with the modular transformation property of $\CC(g, h)$ proved in Section \ref{cbdefn}.

The rest of the paper is devoted to examples.

In Section \ref{ex1} we consider the neutral free fermion VOSA $V = F(\varphi)$ and we take $G = \{1, \sigma_V\}$. We explicitly compute conformal blocks $\CC(g, h; u)$ for $u = \vac$ and $u = \varphi$ the weight $1/2$ vector. In weight $1/2$ the {\sf Type II} supertrace function makes an appearance; the corresponding space of conformal blocks is one dimensional and is spanned by Dedekind's weight $1/2$ modular form $\eta(\tau)$.

In Section \ref{ex2} we study the charged free fermions VOSA $V = F_{\text{ch}}^a(\psi, \psi^*)$ which is equipped with a Virasoro field $L^a(z)$ depending on the real parameter $a \in (0, 1)$. This VOSA admits a group $K \cong S^1$ of automorphisms. For each $g, h \in K$ we write down the supertrace of $h$ on the unique irreducible $g$-twisted positive energy $V$-module, and we compute transformations of these functions under $SL_2(\Z)$. The results confirm Theorem \ref{mythm} which applies when $a \in \Q$ and $g, h$ have finite order.

In Section \ref{ex3} we study the VOSA $V_Q$ associated to a positive definite integral (not necessarily even) lattice $Q$. For $G = \{1, \sigma_V\}$ we describe the $1$- and $\sigma_V$-twisted modules and the spaces $\CC(g, h; \vac)$.

Finally we discuss here the example of the affine Kac-Moody VOA $V = V_k(\mathfrak{sl}_2)$ where $k > -2$ is a rational number (see \cite{Kac}, \cite{DLMadmiss}). This VOA admits a family of `perturbed' VOA structures depending on a parameter $z \in (0, 1) \cap \Q$. This family was studied in detail in \cite{DLMadmiss} (see also \cite{AM}) and it was shown there that $V(z)$ possesses rational (but not integer) conformal weights, that $V(z)$ is $g_0$-rational for a certain finite order automorphism $g_0$ (which depends on $z$), and that $V(z)$ satisfies the weakened $C_2$-cofiniteness condition of Remark \ref{C2remark} (in the case $g = g_0$ and arbitrary $h \in G = \left<g_0\right>$). It follows then from Theorem \ref{mythm} that the trace functions $\tr_M u_0 q^{L_0-\mathfrak{c}/24}$, as $M$ runs through the finitely many irreducible positive energy $V(z)$-modules, span $\CC(g_0, 1; u)$. This space is invariant under the congruence subgroup $\G_0(N) \subseteq SL_2(\Z)$ of matrices $\slmat$ satisfying $\text{\sf a} \equiv \text{\sf d} \equiv 0 \pmod N$ and $\text{\sf b} \equiv \text{\sf c} \equiv 0 \pmod N$, where $N$ is the order of $G$. More refined results are possible, which we would like to deal with in future work.

%(pg. 78 and Theorem 4.10 of \cite{DLMadmiss})

I would like to thank my Ph.D. advisor Victor Kac for many useful discussions and for reading the manuscript. I would also like to thank the referees for detailed comments which improved the exposition, particularly in Section \ref{symmfunc}.

%%%%%%%%%%%%%%%%%%%%%%%%%%%%%%%%%%%%%%%%%%%%%%%%%%%%%%%%%%%%%%%%%%%%%%%%%%%%%%%%%%%%%%%%
%%%%%%%%%%%%%%%%%%%%%%%%%%%%%%%%%%%%%%%%%%%%%%%%%%%%%%%%%%%%%%%%%%%%%%%%%%%%%%%%%%%%%%%%
%%%%%%%%%%%%%%%%%%%%%%%%%%%%%%%%%%%%%%%%%%%%%%%%%%%%%%%%%%%%%%%%%%%%%%%%%%%%%%%%%%%%%%%%
%%%%%%%%%%%%%%%%%%%%%%%%%%%%%%%%%%%%%%%%%%%%%%%%%%%%%%%%%%%%%%%%%%%%%%%%%%%%%%%%%%%%%%%%
%%%%%%%%%%%%%%%%%%%%%%%%%%%%%%%%%%%%%%%%%%%%%%%%%%%%%%%%%%%%%%%%%%%%%%%%%%%%%%%%%%%%%%%%
%%%%%%%%%%%%%%%%%%%%%%%%%%%%%%%%%%%%%%%%%%%%%%%%%%%%%%%%%%%%%%%%%%%%%%%%%%%%%%%%%%%%%%%%

\section{Basic definitions} \label{definitions}

We use the notation $\Z_+ = \{0, 1, 2, \ldots\}$. All vector spaces and superspaces are over $\C$.

A vector \emph{superspace} $U$ is a vector space graded by $\Z/2\Z = \{\ov{0}, \ov{1}\}$. We call $U_{\ov{0}}$ and $U_{\ov{1}}$ the even and odd components of $U$ respectively. A linear map between vector superspaces is always $\Z/2\Z$-graded, and a subspace is always $\Z/2\Z$-graded. We use the following notations: $p(u) = \alpha$ for homogeneous $u \in U_\al$, and $p(u, v) = (-1)^{p(u)p(v)}$. Every vector superspace $U$ carries a natural involution $\sigma_U$ defined by $\sigma_U(u) = (-1)^{p(u)}$. We write $\C^{m|k}$ for the vector superspace with a basis consisting of $m$ even vectors and $k$ odd vectors.

An \emph{associative superalgebra} is a vector superspace with an associative algebra structure compatible with the $\Z/2\Z$-grading. A homomorphism of superalgebras is a homomorphism of algebras that preserves the $\Z/2\Z$-grading. Isomorphism and automorphism are now defined as usual. We deal with unital superalgebras in this paper. The unit element must be even and a homomorphism of unital superalgebras must therefore be even. A basic example of an associative superalgebra is $\en U$, with $(\en U)_\al = \{X \in \en U | XU_\beta \subseteq U_{\al+\beta}\}$, for $U$ a vector superspace. A \emph{module} over an associative superalgebra $A$ is a vector superspace $M$ together with a homomorphism $A \rightarrow \en M$. Two $A$-modules are \emph{equivalent} if there is a $\Z/2\Z$-graded linear isomorphism between them which intertwines with the $A$-action, the isomorphism may be even or odd.

An associative superalgebra is said to be \emph{simple} if it has no proper $\Z/2\Z$-graded ideals apart from $\{0\}$, and is said to be \emph{semisimple} if it is isomorphic to a direct sum of simple superalgebras.

The commutator of operators $X$ and $Y$ on a vector superspace is defined to be $[X, Y] = XY - p(X, Y) YX$. The \emph{supertrace} of an operator $X \in \en U$ is $\str_U X = \tr_{U_{\ov{0}}} X - \tr_{U_{\ov{1}}} X$. In general $\str_U [X, Y] = 0$.

We write $U[z]$ for the ring of polynomials in $z$ with coefficients in $U$, $U[[z]]$ for the ring of formal power series, and $U((z))$ for the ring of Laurent series, i.e., expressions $\sum_{n \in \Z} a_n z^n$ in which finitely many negative powers of $z$ occur. The space of formal distributions $U[[z^{\pm 1}]]$ is the set of expressions $\sum_{n \in \Z} a_n z^n$ with no restriction on the coefficients $a_n$. Finally $U\{\{z\}\} = \oplus_{r \in \R} z^r U[[z^{\pm 1}]]$. Extension to several variables is straightforward, but note that $U((z))((w)) \neq U((w))((z))$.

We write $\partial_z f(z)$ for the $z$-derivative of $f(z)$, also $z^{(n)}$ for $z^n / n!$, and $[z^n] : f(z)$ for $f_n$ the $z^n$ coefficient of $f(z)$.

A convenient index convention for formal distributions is $f(z) = \sum_{n \in \Z} f_{(n)} z^{-n-1}$.  The \emph{formal residue} operation $\res_z(\cdot)dz : U[[z^{\pm 1}]] \rightarrow U$ is defined by $\res_z f(z) dz = [z^{-1}] : f(z) = f_{(0)}$. We have
\[
\res_z \partial_z f(z) dz = 0 \quad \text{and} \quad \res_z f(z) \partial_z g(z) dz = -\res_z g(z) \partial_z f(z) dz.
\]
Let $f(z) \in U((z))$ and $g(w) \in w\C^\times + w^2 \C[[w]]$. The substitution of $z = g(w)$ into $f(z)$ gives a well-defined element $f(g(w)) \in U((w))$, and we have the formal change of variable formula
\[
\res_z f(z) dz = \res_w f(g(w)) \partial_w g(w) dw.
\]

The \emph{formal delta function} $\delta(z, w) \in \C[[z^{\pm 1}, w^{\pm 1}]]$ is defined by
\[
\delta(z, w) = \sum_{n \in \Z} z^n w^{-n-1}.
\]
The operators
\begin{align*}
i_{z, w} &: U[z^{\pm 1}, w^{\pm 1}, (z-w)^{\pm 1}] \rightarrow U((z))((w)) \\
\text{and} \quad i_{w, z} &: U[z^{\pm 1}, w^{\pm 1}, (z-w)^{\pm 1}] \rightarrow U((w))((z))
\end{align*}
denote expansion of an element as Laurent series in the domains $|z| > |w|$ and $|w| > |z|$, respectively. For example,
\[
i_{z, w} (z-w)^{-1} = \sum_{j \in \Z_+} z^{-j-1} w^j \quad \text{and} \quad i_{w, z} (z-w)^{-1} = -\sum_{j \in \Z_+} z^j w^{-j-1}.
\]
An $\en U$-valued formal distribution $f(z)$ is called a \emph{quantum field} if $f(z)u \in U((z))$ for each $u \in U$.

For definitions regarding vertex operator superalgebras we follow the book of Kac \cite{Kac}.
\vspace{.3cm}
\begin{defn} \label{VOSAdefn}
A \emph{vertex operator superalgebra (VOSA)} is a quadruple $(V, \vac, \om, Y)$ where $V$ is a vector superspace, $\vac$ and $\om$ are even elements of $V$ called the \emph{vacuum vector} and the \emph{Virasoro vector} respectively, and $Y : V \rightarrow (\en V)[[z^{\pm 1}]]$ is an injective linear map such that $Y(u, z)$ is a quantum field for each $u \in V$. The map $Y$ is called the \emph{state-field correspondence}, and is written
\[
Y(u, z) = \sum_{n \in \Z} u_{(n)} z^{-n-1}.
\]
The operators $u_{(n)}$ are called the \emph{Fourier modes} of $u$, and the operation $\cdot_{(n)}\cdot : V \otimes V \rightarrow V$ is called the \emph{$n^\text{th}$ product}. The following axioms are to be satisfied.
\begin{itemize}
\item $Y(\vac, z) = I_V$.

\item For all $u, v \in V$, $n \in \Z$,
\begin{align*}
& \sum_{j \in \Z_+} Y(u_{(n+j)}v, w) \partial_w^{(j)} \delta(z, w) \\
& \phantom{\lim} = Y(u, z) Y(v, w) i_{z, w} (z-w)^n - p(u, v) Y(v, w) Y(u, z) i_{w, z} (z-w)^n.
\end{align*}

\item If $Y(\om, z) = L(z) = \sum_{n \in \Z} L_n z^{-n-2}$ then the operators $L_n$ satisfy the commutation relations of the Virasoro algebra:
\[
[L_m, L_n] = (m-n) L_{m+n} + \delta_{m, -n}\frac{m^3-m}{12} \mathfrak{c},
\]
where $\mathfrak{c} \in \R$ is called the \emph{central charge} of $V$.

\item $L_0$ is diagonalizable on $V$ with rational eigenvalues, and the eigenspaces are finite dimensional. The $L_0$-eigenvalue of an eigenvector $u \in V$ is called the \emph{conformal weight} $\D_u$ of $u$. Also $\D_{\vac} = 0$ and $\D_\om = 2$.

\item The set of conformal weights of $V$ is bounded below.

\item $Y(L_{-1}u, z) = \partial_z Y(u, z)$ for all $u \in V$.
\end{itemize}
\end{defn}

Let $V_k = \{u \in V | \D_u = k\}$. A convenient indexing of the modes, called the conformal weight indexing, is defined by $u_n = u_{(n+\D_u-1)}$ (for $u$ of homogeneous conformal weight, then extended to all $u \in V$ linearly). Hence
\[
Y(u, z) = \sum_{n \in -[\D_u]} u_n z^{-n-\D_u},
\]
where here and below $[\al]$ denotes the coset $\al + \Z$ of $\al \in \R$ modulo $\Z$.

The second axiom of Definition \ref{VOSAdefn} is called the \emph{Borcherds identity}. Expressed in terms of modes it becomes
\begin{align} \label{borcherds}
\begin{split}
& \sum_{j \in \Z_+} \binom{m+\D_u-1}{j} (u_{(n+j)}v)_{m+k}x \\
& \phantom{\lim}= \sum_{j \in \Z_+} (-1)^j \binom{n}{j} \left[u_{m+n-j} v_{k+j-n} - p(u, v) (-1)^n v_{k-j}u_{m+j}\right] x
\end{split}
\end{align}
for all $u, v, x \in V$, $n \in \Z$, $m \in -[\D_u]$, and $k \in -[\D_v]$.

A useful special case of the Borcherds identity is the \emph{commutator formula}
\begin{align} \label{commutator}
[u_m, v_k] = \sum_{j \in \Z_+} \binom{m+\D_u-1}{j} (u_{(j)}v)_{m+k},
\end{align}
obtained by setting $n = 0$ in (\ref{borcherds}). The commutator formula together with the final VOSA axiom implies that $[L_0, u_k] = -ku_k$ for all $u \in V$.

A homomorphism $\phi :V_1 \rightarrow V_2$ of VOSAs is an even linear map such that $\phi(\vac_1) = \vac_2$, $\phi(\om_1) = \om_2$, and $\phi(u_{(n)}v) = \phi(u)_{(n)}\phi(v)$ for all $u, v \in V_1$. Homomorphisms preserve conformal weight. Isomorphism and automorphism are defined in the obvious way.

\vspace{.3cm}
\begin{defn} \label{twisteddefn}
Let $V$ be a VOSA and $g$ an automorphism of $V$. Let $\mu(u)$ denote the $g$-eigenvalue of an eigenvector $u \in V$. Pull $\mu(u)$ back to a coset $[\eps_u]$ in $\R$ modulo $\Z$ via the map $e^{2\pi i x} : \R \rightarrow S^1$ (also define $\eps_u$ to be the largest non positive element of $[\eps_u]$). A \emph{$g$-twisted $V$-module} is a vector superspace $M$ together with a state-field correspondence $Y^M : V \rightarrow (\en M)\{\{z\}\}$,
\[
Y^M(u, z) = \sum_{n \in [\eps_u]} u^M_{n} z^{-n-\D_u} = \sum_{n \in [\eps_u] + [\D_u]} u^M_{(n)} z^{-n-1},
\]
satisfying the quantum field property. The following axioms are to be satisfied.
\begin{itemize}
\item $Y^M(\vac, z) = I_M$.

\item For all $u, v \in V$, $x \in M$, $n \in \Z$, $m \in [\eps_u]$, and $k \in [\eps_v]$,
\begin{align*}
& \sum_{j \in \Z_+} \binom{m+\D_u-1}{j} (u_{(n+j)}v)^M_{m+k}x \\
& \phantom{\lim} = \sum_{j \in \Z_+} (-1)^j \binom{n}{j} \left[u^M_{m+n-j} v^M_{k+j-n} - p(u, v) (-1)^n v^M_{k-j}u^M_{m+j}\right] x.
\end{align*}
\end{itemize}
A \emph{positive energy $g$-twisted $V$-module} is a $g$-twisted $V$-module $M$ such that
\begin{itemize}
\item $M = \oplus_{j \in \R} M_j$ is $\R$-graded, each graded piece is finite dimensional, and $M_k = 0$ for sufficiently negative $k$,

\item $u^M_n M_j \subseteq M_{j-n}$ for all $u \in V$, $n \in [\eps_u]$, $j \in \R_+$.
\end{itemize}
\end{defn}

There is a `dangerous bend' in Definition \ref{twisteddefn}. If $u \in V$ has conformal weight $\D_u \notin \Z$ then the modes $u_n$ acting on $V$ are indexed by $n \in -[\D_u]$, so $V$ with its adjoint action is an $e^{-2\pi i L_0}$-twisted $V$-module, \emph{not} a $1$-twisted $V$-module as one might expect. This issue is purely notational, and we could change notation so as to have $V$ be a $1$-twisted $V$-module. We use Definition \ref{twisteddefn} because it is most natural in relation to the modular transformations of conformal blocks, equation (\ref{whereCBgo}). Our definition coincides with the usual one when all conformal weights of $V$ are integers.

\vspace{.3cm}
\begin{defn} \label{Zhustructure}
The \emph{Zhu VOSA structure} is $(V, \vac, \tilde{\om}, Y[u, z])$ where
\[
\tilde{\om} = (2\pi i)^2 (\om - \tfrac{\mathfrak{c}}{24} \vac)
\]
and
\[
Y[u, z] = e^{2\pi i \D_u z} Y(u, e^{2\pi i z} - 1).
\]
\end{defn}
If we write $L[z] = Y[\tilde{\om}, z] = \sum_{n \in \Z} L_{[n]} z^{-n-2}$ then
\begin{align*}
L_{[-2]} = (2\pi i)^2 (L_{-2} - \mathfrak{c} / 24), \quad L_{[-1]} = 2\pi i (L_{-1} + L_0), \quad \text{and} \quad L_{[0]} = L_0 - \sum_{j \in \Z_{>0}} \frac{(-1)^j}{j(j+1)} L_j.
\end{align*}
The eigenvalue $\DD_u$ of an eigenvector $u$ with respect to $L_{[0]}$ is called the \emph{Zhu weight} of $u$. We write $V_{[k]} = \{u \in V | \DD_u = k\}$ and
\[
Y[u, z] = \sum_{n \in \Z} u_{([n])} z^{-n-1} = \sum_{n \in -[\DD_u]} u_{[n]} z^{-n-\DD_u}.
\]
Explicitly we have
\begin{align*}
u_{([n])}v
&= \res_z z^n e^{2\pi i \D_u z} Y(u, e^{2\pi i z} - 1)v dz \\
&= (2\pi i)^{-n-1} \res_w [\ln(1+w)]^n (1+w)^{\D_u-1} Y(u, w)v dw,
\end{align*}
where $w = e^{2\pi i z}-1$. An automorphism of the new VOSA structure is the same as an automorphism of the old one. Vectors of homogeneous conformal weight are not generally of homogeneous Zhu weight and vice versa.

We use the following notation below: $V$ is a VOSA, $G$ a finite group of automorphisms of $V$, and $g, h \in G$ two commuting automorphisms. Unless otherwise stated an element of $V$ is a simultaneous eigenvector of $g$ and $h$. For such an eigenvector $u$ we write $\mu(u)$ and $\la(u)$ for its $g$- and $h$-eigenvalues respectively.

We define a right action of $SL_2(\Z)$ on $G \times G$ by $(g, h) \cdot A = (g^{\sf a} h^{\sf c}, g^{\sf b} h^{\sf d})$ where $A = \slmat$. Similarly $(\mu, \la) \cdot A = (\mu^{\sf a} \la^{\sf c}, \mu^{\sf b}\la^{\sf d})$. We use the standard notation $A\tau$ for $\frac{{\sf a}\tau + {\sf b}}{{\sf c}\tau + {\sf d}}$.

%%%%%%%%%%%%%%%%%%%%%%%%%%%%%%%%%%%%%%%%%%%%%%%%%%%%%%%%%%%%%%%%%%%%%%%%%%%%%%%%%%%%%%%%
%%%%%%%%%%%%%%%%%%%%%%%%%%%%%%%%%%%%%%%%%%%%%%%%%%%%%%%%%%%%%%%%%%%%%%%%%%%%%%%%%%%%%%%%
%%%%%%%%%%%%%%%%%%%%%%%%%%%%%%%%%%%%%%%%%%%%%%%%%%%%%%%%%%%%%%%%%%%%%%%%%%%%%%%%%%%%%%%%
%%%%%%%%%%%%%%%%%%%%%%%%%%%%%%%%%%%%%%%%%%%%%%%%%%%%%%%%%%%%%%%%%%%%%%%%%%%%%%%%%%%%%%%%
%%%%%%%%%%%%%%%%%%%%%%%%%%%%%%%%%%%%%%%%%%%%%%%%%%%%%%%%%%%%%%%%%%%%%%%%%%%%%%%%%%%%%%%%
%%%%%%%%%%%%%%%%%%%%%%%%%%%%%%%%%%%%%%%%%%%%%%%%%%%%%%%%%%%%%%%%%%%%%%%%%%%%%%%%%%%%%%%%

\section{Modular Forms} \label{modforms}

In this section we recall some functions that appear in connection with modular forms and elliptic curves. Consider the ill-defined expression
\[
2\pi i {\sum_{n \in [\eps]}}' \frac{e^{2\pi i n z}}{1 - \la q^n},
\]
where $\la$ is a root of unity and $[\eps]$ is a coset of $\Q$ modulo $\Z$ (also fix $\eps \in [\eps]$ such that $-1 < \eps \leq 0$, and let $\mu = e^{2\pi i \eps}$). By ${\sum}'$ we mean the summation over all nonsingular terms, i.e., if $[\eps] = \Z$ and $\la = 1$ then $n = 0$ is to be excluded from the sum.

To make sense of the sum we first rewrite it as
\begin{align*}
\frac{2\pi i \delta}{1-\la} + 2\pi i \sum_{n \in [\eps]_{> 0}} \frac{e^{2\pi i n z}}{1 - \la q^n} - 2\pi i \sum_{n \in [\eps]_{< 0}} \frac{\la^{-1} e^{2\pi i n z} q^{-n}}{1 - \la^{-1} q^{-n}},
\end{align*}
where
\[
\delta = \left\{\begin{array}{ll} 1 & \text{if $[\eps] = \Z$ and $\la \neq 1$}, \\ 0 & \text{otherwise}. \\ \end{array} \right.
\]
Then we expand in non-negative powers of $q$ to get
\begin{align*}
\frac{2\pi i \delta}{1-\la}
+ 2\pi i \sum_{n \in [\eps]_{>0}} e^{2\pi i n z}
+ 2\pi i \sum_{m \in \Z_{>0}} \left[ \sum_{n \in [\eps]_{> 0}} e^{2\pi i n z} (\la q^n)^m
- \sum_{n \in [\eps]_{< 0}} e^{2\pi i n z} (\la^{-1} q^{-n})^m \right].
\end{align*}
This is still not well-defined, because of the second term. Let us re-sum the second term using the geometric series formula. We arrive at the following formula, which we regard as a definition.
\begin{align} \label{newPdef}
\begin{split}
P^{\mu, \la}(z, q)
= {} & \frac{2\pi i \delta}{1-\la}
- 2\pi i \frac{e^{2\pi i (1+\eps) z}}{e^{2\pi i z} - 1} \\
& + 2\pi i \sum_{m \in \Z_{>0}} \left[ \sum_{n \in [\eps]_{> 0}} e^{2\pi i n z} (\la q^n)^m
- \sum_{n \in [\eps]_{< 0}} e^{2\pi i n z} (\la^{-1} q^{-n})^m \right].
\end{split}
\end{align}

Let us write
\begin{align*}
P^{\mu, \la}(z, q) = -z^{-1} + \sum_{k=0}^\infty P^{\mu, \la}_k(q) z^k.
\end{align*}
The \emph{Bernoulli polynomials} $B_n(\gamma)$ are defined by
\begin{align*}
\frac{e^{\gamma z}}{e^z-1} = \sum_{n=0}^\infty \frac{z^{n-1}}{n!} B_n(\gamma).
\end{align*}
For example $B_0(\gamma) = 1$, $B_1(\gamma) = \gamma - 1/2$, $B_2(\gamma) = \gamma^2 - \gamma + 1/6$, etc. The \emph{Bernoulli numbers} are $B_n = B_n(1)$. Using the definition of the Bernoulli polynomials and the series expansion of $e^{2\pi i n z}$ we directly obtain the following.
\vspace{.3cm}
\begin{lemma} \label{expmod}
For $k \in \Z_+$ we have
\begin{align*}
P_k^{\mu, \la}(q) = {} & \delta_{k, 0} \frac{2\pi i \delta}{1-\la} - \frac{(2\pi i)^{k+1}}{(k+1)!} B_{k+1}(1+\eps) \\
& + \frac{(2\pi i)^{k+1}}{k!} \sum_{m \in \Z_{>0}} \left[ \sum_{n \in [\eps]_{>0}} n^k (\la q^n)^m - \sum_{n \in [\eps]_{<0}} n^k (\la^{-1} q^{-n})^m \right].
\end{align*}
\end{lemma}

We now record the modular transformation properties of the functions $P_k^{\mu, \la}$. For $(\mu, \la) \neq (1, 1)$, our functions are essentially the same as the $Q$-functions of \cite{DLMorbifold}. Indeed for all $k \in \Z_+$,
\begin{align*}
P_k^{\mu, \la}(q) &= (2\pi i)^{k+1} Q_{k+1}(\mu, \la, q) \quad \text{when $(\mu, \la) \neq (1, 1)$}.
\end{align*}
Section 4 of \cite{DLMorbifold}, in particular Theorem 4.6, tells us that $P_k^{\mu, \la}$, when summed in order of increasing powers of $q$, converges to a holomorphic function of $\tau \in \HH$. Furthermore
\begin{align} \label{Pmodular}
P_k^{\mu, \la}(A \tau) = ({\sf c}\tau+{\sf d})^{k+1} P_k^{(\mu, \la) \cdot A}(\tau).
\end{align}
Since $\mu$ and $\la$ are roots of unity, there exists $N \in \Z_+$ such that $\mu^N = \la^N = 1$. Therefore $P_k^{\mu, \la}(A \tau) = ({\sf c}\tau+{\sf d})^{k+1} P_k^{\mu, \la}(\tau)$ whenever $A = \slmat$ satisfies ${\sf a} \equiv {\sf d} \equiv 1 \pmod N$ and ${\sf b} \equiv {\sf c} \equiv 0 \pmod N$, i.e., if $A \in \Gamma_0(N)$. Hence $P_k^{\mu, \la}(\tau)$ is a holomorphic modular form on $\G_0(N)$ of weight $k+1$.

Now we consider the case $(\mu, \la) = (1, 1)$. Comparing the formula of Lemma \ref{expmod} with the Eisenstein series
\begin{align*}
G_k(\tau)
= 2\zeta(k) + \frac{2(2\pi i)^k}{(k-1)!} \sum_{n=1}^\infty \sigma_{k-1}(n) q^n \quad \text{(for $k \geq 2$)}
\end{align*}
shows that $P^{1, 1}_k(q) = G_{k+1}(\tau)$ for $k \geq 1$. We also have $P_0^{1, 1}(q) = -\pi i$. Therefore equation (\ref{Pmodular}) holds when $(\mu, \la) = (1, 1)$ and $k \geq 2$. It is well-known that $G_2(q)$ is not a modular form, but instead satisfies
\begin{align*}
G_2(A \tau) = ({\sf c}\tau+{\sf d})^2 G_2(\tau) - 2\pi i {\sf c} ({\sf c}\tau+{\sf d}).
\end{align*}

The function $P^{1, 1}(z, q)$ (which we abbreviate to $P(z, q)$ below) is closely related to the classical Weierstrass zeta function
\begin{align*}
\zeta(z, \tau)
&= z^{-1} + \sum_{(m, n) \neq (0, 0)} \left[ \frac{1}{z-m\tau-n} + \frac{1}{m\tau+n} + \frac{z}{(m\tau+n)^2} \right] \\
&= z^{-1} - \sum_{k=4}^\infty z^{k-1} G_k(\tau),
\end{align*}
(the latter is the Laurent expansion about $z=0$). We have
\[
\zeta(z, \tau) = -P(z, q) + zG_2(q) - \pi i.
\]
The Weierstrass elliptic function is $\wp(z, \tau) = -\frac{\partial}{\partial z} \zeta(z, \tau)$, so we have
\begin{align} \label{delpwp}
\frac{\partial}{\partial z} P(z, q) = \wp(z, q) + G_2(q).
\end{align}

The Dedekind eta function is defined, for $\tau \in \HH$, to be
\begin{align} \label{etadefn}
\eta(\tau) = q^{1 / 24} \prod_{n=1}^\infty (1-q^n).
\end{align}
From \cite{Lang} p. 253 we have the following.
\vspace{.3cm}
\begin{prop} \label{etatrans}
\[
\eta(\tau+1) = e^{\pi i / 12} \eta(\tau) \quad \text{and} \quad \eta(\tfrac{-1}{\tau}) = (-i \tau)^{1/2} \eta(\tau).
\]
\end{prop}

The Jacobi theta function is defined, for $\tau \in \HH$ and $z \in \C$, to be
\begin{align} \label{Jacdefn}
\theta(z; \tau) = \sum_{n \in \Z} e^{\pi i n^2 \tau + 2\pi i n z}.
\end{align}
%Traditionally $\theta$ is referred to as $\vartheta_3$.
From \cite{WW}, p. 475 we have the following.
\vspace{.3cm}
\begin{prop} \label{Jactrans}
\begin{align*}
\theta(\tfrac{z}{\tau}; \tfrac{-1}{\tau}) = (-i \tau)^{1/2} e^{\pi i z^2 / \tau} \theta(z; \tau) \quad \text{and} \quad \theta(z; \tau+1) = \theta(z + \tfrac{1}{2}; \tau).
\end{align*}
\end{prop}

%%%%%%%%%%%%%%%%%%%%%%%%%%%%%%%%%%%%%%%%%%%%%%%%%%%%%%%%%%%%%%%%%%%%%%%%%%%%%%%%%%%%%%%%
%%%%%%%%%%%%%%%%%%%%%%%%%%%%%%%%%%%%%%%%%%%%%%%%%%%%%%%%%%%%%%%%%%%%%%%%%%%%%%%%%%%%%%%%
%%%%%%%%%%%%%%%%%%%%%%%%%%%%%%%%%%%%%%%%%%%%%%%%%%%%%%%%%%%%%%%%%%%%%%%%%%%%%%%%%%%%%%%%
%%%%%%%%%%%%%%%%%%%%%%%%%%%%%%%%%%%%%%%%%%%%%%%%%%%%%%%%%%%%%%%%%%%%%%%%%%%%%%%%%%%%%%%%
%%%%%%%%%%%%%%%%%%%%%%%%%%%%%%%%%%%%%%%%%%%%%%%%%%%%%%%%%%%%%%%%%%%%%%%%%%%%%%%%%%%%%%%%
%%%%%%%%%%%%%%%%%%%%%%%%%%%%%%%%%%%%%%%%%%%%%%%%%%%%%%%%%%%%%%%%%%%%%%%%%%%%%%%%%%%%%%%%

\section{Conformal Blocks} \label{cbdefn}

Let $\M_{N}$ be the vector space of holomorphic modular forms on $\G_0(N)$, i.e., the vector space of holomorphic functions $f : \HH \rightarrow \C$ such that
\begin{itemize}
\item $f(A \tau) = f(\tau)$ for all $A \in \G_0(N)$.

\item $f(A \tau)$ is meromorphic at $\tau = i\infty$ for all $A \in SL_2(\Z)$.
\end{itemize}

\vspace{.3cm}
\begin{defn} \label{CBdefn}
Let $\V = \M_{|G|} \otimes_\C V$. Define $\OO(g, h)$ to be the $\M_{{|G|}}$-submodule of $\V$ generated by
\begin{align*}
\left\{\begin{array}{ll}
\ide_1(u, v) = \res_z Y[u, z]v dz = u_{([0])}v & \text{for $(\mu(u), \la(u)) = (1, 1)$}, \\
\ide_2(u, v) = \res_z \wp(z, q) Y[u, z]v dz & \text{for $(\mu(u), \la(u)) = (1, 1)$}, \\
u & \text{for $(\mu(u), \la(u)) \neq (1, 1)$}, \\
\ide_3^{g, h}(u, v) = \res_z P^{\mu(u), \la(u)}(z, q) Y[u, z]v dz & \text{for $(\mu(u), \la(u)) \neq (1, 1)$}.
\end{array}\right.
\end{align*}
\end{defn}
\vspace{.3cm}
\begin{defn} \label{CBdefn2}
The space $\CC(g, h)$ of \emph{conformal blocks} is the space of functions $S : \V \times \HH \rightarrow \C$ satisfying
\begin{description}
\item [CB1] $S(x+y, \tau) = S(x, \tau) + S(y, \tau)$ for all $x, y \in \V$, and $S(f(\tau) u, \tau) = f(\tau) S(u, \tau)$ for all $f(\tau) \in \M_{|G|}$, $u \in V$.

\item [CB2] \label{axiom1} $S(u, \tau)$ is holomorphic in $\tau$ for each $u \in V$.

\item [CB3] $S(x, \tau) = 0$ for all $x \in \OO(g, h)$.

\item [CB4] For all $u \in V$ such that $(\mu(u), \la(u)) = (1, 1)$,
\begin{align} \label{ax4}
\left[(2\pi i)^2 q\frac{d}{dq} + \DD_u G_2(q)\right] S(u, \tau) = S(\res_z \zeta(z, q) L[z]u dz, \tau).
\end{align}
\end{description}
\end{defn}

An equivalent form of (\ref{ax4}) is
\begin{align} \label{ax4'}
(2\pi i)^2 q\frac{d}{dq} S(u, \tau) = -S(\res_z P(z, q) L[z]u dz, \tau).
\end{align}

%%%%%%%%%%%%%%%%%%%%%%%%%%%%%%%%%%%%%%%%%%%%%%%%%%%%%%%%%%%%%%%%%%%%%%%%%%%%%%%%%%%%%%%%
%%%%%%%%%%%%%%%%%%%%%%%%%%%%%%%%%%%%%%%%%%%%%%%%%%%%%%%%%%%%%%%%%%%%%%%%%%%%%%%%%%%%%%%%
%%%%%%%%%%%%%%%%%%%%%%%%%%%%%%%%%%%%%%%%%%%%%%%%%%%%%%%%%%%%%%%%%%%%%%%%%%%%%%%%%%%%%%%%
%%%%%%%%%%%%%%%%%%%%%%%%%%%%%%%%%%%%%%%%%%%%%%%%%%%%%%%%%%%%%%%%%%%%%%%%%%%%%%%%%%%%%%%%
%%%%%%%%%%%%%%%%%%%%%%%%%%%%%%%%%%%%%%%%%%%%%%%%%%%%%%%%%%%%%%%%%%%%%%%%%%%%%%%%%%%%%%%%
%%%%%%%%%%%%%%%%%%%%%%%%%%%%%%%%%%%%%%%%%%%%%%%%%%%%%%%%%%%%%%%%%%%%%%%%%%%%%%%%%%%%%%%%

\subsection{Modular transformations of conformal blocks}

Let $K$ be a positive integer such that $1/K$ divides the conformal weight of each vector in $V$ (the $C_2$-cofiniteness condition implies that $K$ exists, see the first paragraph in the proof of Lemma \ref{vofingen} below). Let $\sqrt[K]{z}$ denote the principal $K^\text{th}$ root of $z$ i.e., $-\pi/K < \arg(\sqrt[K]{z}) \leq \pi/K$. In the following theorem $({\sf c}\tau + {\sf d})^{-k}$ is defined as the appropriate integer power of $\sqrt[K]{{\sf c}\tau + {\sf d}}$.
\vspace{.3cm}
\begin{thm} \label{sl2zinv}
Let $S \in \CC(g, h)$ and $A \in SL_2(\Z)$. Define $S \cdot A : \V \times \HH \rightarrow \C$ by
\begin{align*}
[S \cdot A](u, \tau) &= ({\sf c}\tau+{\sf d})^{-k} S(a, A \tau) \quad \text{for $u \in V_{[k]}$}, \\
\text{and} \quad [S \cdot A](f(\tau)u, \tau) &= f(\tau) [S \cdot A](u, \tau) \quad \text{for $u \in V$, $f(\tau) \in \M_{{|G|}}$}.
\end{align*}
Then $S \cdot A \in \CC((g, h) \cdot A)$.
\end{thm}

\begin{proof}
Fix $g, h \in G$, $A \in SL_2(\Z)$, and let $S \in \CC(g, h)$. It is obvious that $S \cdot A$ satisfies {\bf CB1}. Because $S(u, \tau)$ is holomorphic in $\tau$, $S(u, A\tau)$ is too. Because ${\sf c}\HH + {\sf d}$ is disjoint from the branch cut, $({\sf c}\tau+{\sf d})^{-k}$ is holomorphic in $\tau$. Therefore $S \cdot A$ satisfies {\bf CB2}.

Clearly $(\mu(u), \la(u)) = (1, 1)$ if and only if $(\mu(u), \la(u)) \cdot A = (1, 1)$.

Suppose $(\mu(u), \la(u)) = (1, 1)$. We have $S(\ide_1(u, v), \tau) = 0$, hence
\[
[S \cdot A](\ide_1(u, v), \tau) = ({\sf c}\tau+{\sf d})^{-\DD_u-\DD_v-1} S(\ide_1(u, v), A\tau) = 0.
\]

Next we have
\begin{align*}
[S \cdot A](\ide_2(u, v), \tau)
= {} & [S \cdot A](u_{([-2])}v, \tau)
+ \sum_{k=2}^\infty (2k-1) G_{2k}(\tau) [S \cdot A](u_{([2k-2])}v, \tau) \\
= {} & ({\sf c}\tau+{\sf d})^{-\DD_u-\DD_v-1} S(u_{([-2])}v, A \tau) \\
& + \sum_{k=2}^\infty (2k-1) G_{2k}(\tau) ({\sf c}\tau+{\sf d})^{-\DD_u-\DD_v+2k-1} S(u_{([2k-2])}v, A \tau) \\
= {} & ({\sf c}\tau+{\sf d})^{-\DD_u-\DD_v-1} \left[ S(u_{([-2])}v, A \tau) + \sum_{k=2}^\infty (2k-1) G_{2k}(A \tau) S(u_{([2k-2])}v, A \tau) \right] \\
= {} & ({\sf c}\tau+{\sf d})^{-\DD_u-\DD_v-1} S(\ide_2(u, v), A \tau)
= 0.
\end{align*}

Now suppose $(\mu, \la) = (\mu(u), \la(u))_{(g, h)} \neq (1, 1)$ (so $(\mu(u), \la(u))_{(g, h) \cdot A} \neq (1, 1)$ too). We have
\begin{align*}
[S \cdot A](\ide_3^{(g, h) \cdot A}(u, v), \tau)
= {} & [S \cdot A](u_{([-1])}v, \tau)
+ \sum_{k=0}^\infty P^{(\mu, \la) \cdot A}_k(\tau) [S \cdot A](u_{([k])}v, \tau) \\
= {} & ({\sf c}\tau+{\sf d})^{-\DD_u-\DD_v} S(u_{([-1])}v, A \tau) \\
& - ({\sf c}\tau+{\sf d})^{-\DD_u-\DD_v}\sum_{k=0}^\infty P^{(\mu, \la) \cdot A}_k(\tau) ({\sf c}\tau+{\sf d})^{k+1} S(u_{([k])}v, A \tau) \\
= {} & ({\sf c}\tau+{\sf d})^{-\DD_u-\DD_v} \left[ S(u_{([-1])}v, A \tau) - \sum_{k=0}^\infty P^{\mu, \la}_k(A\tau) S(u_{([k])}v, A \tau) \right] \\
= {} & ({\sf c}\tau+{\sf d})^{-\DD_u-\DD_v} S(\ide_3^{g, h}(u, v), A \tau)
= 0
\end{align*}
(having used the transformation property (\ref{Pmodular}) of $P_k^{\mu, \la}$). Finally note that $[S \cdot A](u, \tau) = 0$ whenever $(\mu(u), \la(u)) \neq (1, 1)$ because the same is true for $S$. Thus $S \cdot A$ satisfies {\bf CB3}.

Let $(\mu(u), \la(u)) = (1, 1)$ again. By a calculation similar to the one above, we have
\begin{align*}
[S \cdot A](\res_z \zeta(z, \tau) L[z]u dz, \tau)
= ({\sf c}\tau+{\sf d})^{-\DD_u-2} S(\res_z \zeta(z, A\tau) L[z]u dz, A\tau).
\end{align*}
On the other hand
\begin{align*}
& \left[2\pi i \frac{d}{d\tau} + \DD_u G_2(\tau)\right] [S \cdot A](u, \tau) \\
= {} & \left[2\pi i \frac{d}{d\tau} + \DD_u G_2(\tau)\right] ({\sf c}\tau+{\sf d})^{-\DD_u} S(u, A\tau) \\
= {} & \left[ -2\pi i {\sf c} \DD_u ({\sf c}\tau+{\sf d})^{-\DD_u-1}
+ ({\sf c}\tau+{\sf d})^{-\DD_u} \frac{d(A\tau)}{d\tau} \frac{d}{d(A\tau)}
+ \DD_u G_2(\tau) ({\sf c}\tau+{\sf d})^{-\DD_u} \right] S(u, A\tau) \\
= {} & ({\sf c}\tau+{\sf d})^{-\DD_u-2} \left[ -2\pi i {\sf c} \DD_u ({\sf c}\tau+{\sf d}) + \frac{d}{d(A\tau)}
+ \DD_u G_2(\tau) ({\sf c}\tau+{\sf d})^2 \right] S(u, A\tau) \\
= {} & ({\sf c}\tau+{\sf d})^{-\DD_u-2} \left[ \frac{d}{d(A\tau)} + \DD_u G_2(A\tau) \right] S(u, A\tau).
\end{align*}
So $S \cdot A$ satisfies {\bf CB4}.
\end{proof}

%%%%%%%%%%%%%%%%%%%%%%%%%%%%%%%%%%%%%%%%%%%%%%%%%%%%%%%%%%%%%%%%%%%%%%%%%%%%%%%%%%%%%%%%
%%%%%%%%%%%%%%%%%%%%%%%%%%%%%%%%%%%%%%%%%%%%%%%%%%%%%%%%%%%%%%%%%%%%%%%%%%%%%%%%%%%%%%%%
%%%%%%%%%%%%%%%%%%%%%%%%%%%%%%%%%%%%%%%%%%%%%%%%%%%%%%%%%%%%%%%%%%%%%%%%%%%%%%%%%%%%%%%%
%%%%%%%%%%%%%%%%%%%%%%%%%%%%%%%%%%%%%%%%%%%%%%%%%%%%%%%%%%%%%%%%%%%%%%%%%%%%%%%%%%%%%%%%
%%%%%%%%%%%%%%%%%%%%%%%%%%%%%%%%%%%%%%%%%%%%%%%%%%%%%%%%%%%%%%%%%%%%%%%%%%%%%%%%%%%%%%%%
%%%%%%%%%%%%%%%%%%%%%%%%%%%%%%%%%%%%%%%%%%%%%%%%%%%%%%%%%%%%%%%%%%%%%%%%%%%%%%%%%%%%%%%%

\section{Differential equations satisfied by conformal blocks} \label{ODEsection}

We recall the crucial \emph{$C_2$-cofiniteness} condition introduced in \cite{Zhu}. This condition, together with the conformal block axioms, implies the existence of an ordinary differential equation (ODE) satisfied by the conformal blocks.

\vspace{.3cm}
\begin{defn} \label{C2cofinite}
The vertex operator algebra $V$ is said to be \emph{$C_2$-cofinite} if the subspace
\[
C_2(V) = \spa \{u_{(-2)}v | u, v \in V\} \subseteq V
\]
has finite codimension in $V$.
\end{defn}

\vspace{.3cm}
\begin{lemma} \label{vofingen}
If $V$ is $C_2$-cofinite then the $\M_{|G|}$-module $\V / \OO(g, h)$ is finitely generated, for each $g, h \in G$.
\end{lemma}

\begin{proof}
Since $C_2(V)$ is a graded subspace of $V$ (under the $\D$-grading) there exists $n_0 \in \Z_+$ such that $V_n \subseteq C_2(V)$ for all $n > n_0$. Let $W = \oplus_{k \leq n_0} V_k \subseteq V$.

Since $\D_{u_{(-2)}v} = \D_u + \D_v + 1$, every vector in $V$ with conformal weight greater than $n_0$ can be expressed in terms of $-2^\text{nd}$ products of vectors in $W$. Therefore all conformal weights in $V$ are integer multiples of $1/K$ for some positive integer $K$.

Let $\W = \M_{|G|} W \subseteq \V$. Recall that
\begin{align*}
L_{[0]} = L_{0} + \sum_{i \geq 1} \al_{0i} L_{i}
\quad \text{and} \quad
L_{0} = L_{[0]} + \sum_{i \geq 1} \beta_{0i} L_{[i]},
\end{align*}
for certain $\al_{0i}, \beta_{0i} \in \C$. Suppose $u \in V_n$, i.e., $L_0 u = nu$. Then $L_{[0]}u = nu$ modulo terms with strictly lower $\D$. Similarly if $v \in V_{[n]}$ then $L_0 v = nv$ modulo terms with strictly lower $\DD$. Thus $\oplus_{k \leq n} V_k = \oplus_{k \leq n} V_{[k]}$ for any $n \in \Q$.

We will prove by induction on conformal weight (which is possible since conformal weights are multiples of $1/K$ and are bounded below) that $V_{[n]} \subseteq \W + \OO(g, h)$ for all $n$. According to the last paragraph this holds for $n \leq n_0$ already.

Let $n > n_0$, and let $x \in V_{[n]}$. Since $V = W + C_2(V)$ we may write $x$ as $w \in W$ plus a sum of vectors of the form
\[
u_{(-2)}v = u_{([-2])}v + \sum_{j > -2} \al_{-2, j} u_{([j])}v,
\]
where we assume $u, v$ are homogeneous in the $\DD$-grading. It is clear that we can choose all the pairs of vectors $u, v$ so that $\DD_u + \DD_v + 1 \leq n$. Therefore all the terms in the $j$-summation have $\DD < n$, hence they lie in $\W + \OO(g, h)$ by the inductive assumption. It suffices to show that $u_{([-2])}v \in \W + \OO(g, h)$.

If $(\mu(u), \la(u)) = (1, 1)$ then
\[
\ide_2(u, v) = u_{([-2])}v + \sum_{k=2}^\infty (2k-1) G_{2k}(\tau) u_{([2k-2])}v \in \OO(g, h).
\]
The terms in the summation have $\DD < n$, hence they lie in $\W + \OO(g, h)$ by the inductive assumption. Therefore $u_{([-2])}v$ does too.

If $(\mu(u), \la(u)) \neq (1, 1)$ then
\[
\ide_3(u, v) = -u_{([-1])}v + \sum_{k=0}^\infty P_k^{\mu(u), \la(u)}(q) u_{([k])}v \in \OO(g, h).
\]
Substituting $L_{[-1]}u$ in place of $u$ shows that
\[
u_{([-2])}v - \sum_{k=0}^\infty k P_k^{\mu(u), \la(u)}(q) u_{([k-1])}v \in \OO(g, h)
\]
too. As before, $u_{([-2])}v \in \W + \OO(g, h)$.
\end{proof}

\vspace{.3cm}
\begin{rmrk} \label{C2remark} {\ }
Inspection of the proof of Lemma \ref{vofingen} reveals that the $C_2$-cofiniteness condition can be weakened to the following: $V / C^{(g, h)}$ is finite dimensional where $C^{(g, h)}$ is defined to be the span of the vectors
\[
u_{(-2)}v \,\,\, \text{for $(\mu(u), \la(u)) = (1, 1)$}, \quad \text{and} \quad u_{(-1)}v \,\,\, \text{for $(\mu(u), \la(u)) \neq (1, 1)$}.
\]
Therefore all results of this paper hold with $C_2$-cofiniteness replaced by this weaker condition.
\end{rmrk}

The following lemma is stated in \cite{DLMorbifold}.
\vspace{.3cm}
\begin{lemma} \label{Mnoeth}
For any integer $N \geq 1$, $\M_N$ is a Noetherian ring.
\end{lemma}

\vspace{.3cm}
\begin{lemma} \label{Lm2reduc}
Let $V$ be $C_2$-cofinite, let $u \in V$, and let $S \in \CC(g, h)$. There exists $m \in \Z_+$ and $r_0(\tau), \ldots r_{m-1}(\tau) \in \M_{|G|}$ such that
\begin{align} \label{Sfin}
S(L_{[-2]}^m u, \tau) + \sum_{i=0}^{m-1} r_i(\tau) S(L_{[-2]}^i u, \tau) = 0.
\end{align}
\end{lemma}

\begin{proof}
Let $I_n(u) \subseteq \V/\OO(g, h)$ be the $\M_{|G|}$-submodule generated over $\M_{|G|}$ by the images of $u, L_{[-2]}u, \ldots L_{[-2]}^n u$. Because $\V/\OO(g, h)$ is a finitely generated $\M_{|G|}$-module and $\M_{|G|}$ is a Noetherian ring, $\V/\OO(g, h)$ is a Noetherian $\M_{|G|}$-module: meaning that the ascending chain $I_0(u) \subseteq I_1(u) \subseteq \ldots$ stabilizes. For some $m \in \Z_+$ we have $I_m(u) = I_{m-1}(u)$, which implies
\begin{align*}
L_{[-2]}^m u + \sum_{i=0}^{m-1} r_i(\tau) L_{[-2]}^i u \in \OO(g, h)
\end{align*}
for some $r_i(\tau) \in \M_{|G|}$. Equation (\ref{Sfin}) follows from this formula and {\bf CB3}.
\end{proof}

The axiom {\bf CB4} states an equality between $S(L_{[-2]}u, \tau)$ and $(q \frac{d}{dq}) S(u, \tau)$ modulo `terms of lower order', i.e., terms of the form $S(v, \tau)$ with $\DD_v < \DD_u$. We use this to convert (\ref{Sfin}) into an ODE satisfied by $S(u, \tau)$. More precisely we have

\vspace{.3cm}
\begin{thm} \label{getdiff}
Let $\oq = q^{1/{|G|}}$, and let $S \in \CC(g, h)$. For each $u \in V$, $S(u, \tau)$ satisfies an ODE of the form
\[
(\oq\frac{d}{d\oq})^m S(u, \tau) + \sum_{i=0}^{m-1} g_i(\oq) (\oq\frac{d}{d\oq})^i S(u, \tau) + \sum_{j=0}^{m-1} \sum_{k=0}^{\infty} h_{jk}(\oq) (\oq\frac{d}{d\oq})^j S(x_{jk}, \tau) = 0,
\]
where the $x_{jk} \in V$ are of strictly lower conformal weight than $u$ and the functions $g_i(\oq)$ and $h_{jk}(\oq)$ are polynomials in elements of $\M_{|G|}$ and derivatives of $G_2$ with respect to $\oq$. In particular these functions are all regular at $\oq = 0$, and so the ODE has a regular singular point there.
\end{thm}
We write $\oq$ here instead of $q$ because the elements of $\M_{|G|}$ can be expressed as series in integer powers of $\oq$, rather than $q$. For the proof of Theorem \ref{getdiff} see Section 6 of \cite{DLMorbifold}.

For $u^{(0)} \in V$ of minimal conformal weight the ODE satisfied by $S(u^{(0)}, \tau)$ is homogeneous because there are no nonzero vectors with strictly lower conformal weight. The theory of Frobenius-Fuchs tells us that $S(u^{(0)}, \tau)$ may be expressed in a certain form (\ref{canonform}) below. For arbitrary $u \in V$ the same conclusion cannot be drawn directly because of the presence of the inhomogeneous term. However an induction on $\DD_u$ shows that $S(u, \tau)$ does take the form (\ref{canonform}) for all $u \in V$. The form in question is
\begin{align} \label{canonform}
\begin{split}
S(u, \tau) &= \sum_{i=0}^p (\log q)^i S_i(u, \tau), \\
\text{where} \quad S_i(u, \tau) & = \sum_{j=1}^{b(i)} q^{\lambda_{ij}} S_{ij}(u, \tau), \\
\text{where} \quad S_{ij}(u, \tau) & = \sum_{n=0}^{\infty} C_{i, j, n}(u) q^{n/{|G|}},
\end{split}
\end{align}
where $\lambda_{i j_1} - \lambda_{i j_2} \notin \tfrac{1}{|G|}\Z$ for $1 \leq j_1 \neq j_2 \leq b(i)$. We call (\ref{canonform}) the \emph{Frobenius expansion} of $S(u, \tau)$.

A priori the parameters $p$, $b(i)$, and $\la_{i, j}$ in the Frobenius expansion of $S(u, \tau)$ depend on $u$. However if $\{u^{(i)}\}$ is a basis of $W$ then the conformal blocks $S(u^{(i)}, \tau)$ obviously span $\CC(g, h)$. This implies that $\CC(g, h)$ is finite dimensional, and that the Frobenius expansion of $S(u, \tau)$ may be written with fixed $p$, $b(i)$, $\la_{i, j_r}$ independent of $u$.

%%%%%%%%%%%%%%%%%%%%%%%%%%%%%%%%%%%%%%%%%%%%%%%%%%%%%%%%%%%%%%%%%%%%%%%%%%%%%%%%%%%%%%%%
%%%%%%%%%%%%%%%%%%%%%%%%%%%%%%%%%%%%%%%%%%%%%%%%%%%%%%%%%%%%%%%%%%%%%%%%%%%%%%%%%%%%%%%%
%%%%%%%%%%%%%%%%%%%%%%%%%%%%%%%%%%%%%%%%%%%%%%%%%%%%%%%%%%%%%%%%%%%%%%%%%%%%%%%%%%%%%%%%
%%%%%%%%%%%%%%%%%%%%%%%%%%%%%%%%%%%%%%%%%%%%%%%%%%%%%%%%%%%%%%%%%%%%%%%%%%%%%%%%%%%%%%%%
%%%%%%%%%%%%%%%%%%%%%%%%%%%%%%%%%%%%%%%%%%%%%%%%%%%%%%%%%%%%%%%%%%%%%%%%%%%%%%%%%%%%%%%%
%%%%%%%%%%%%%%%%%%%%%%%%%%%%%%%%%%%%%%%%%%%%%%%%%%%%%%%%%%%%%%%%%%%%%%%%%%%%%%%%%%%%%%%%

\section{Coefficients of Frobenius expansions} \label{coeffs}

In this section we study the coefficients $C_{p, j, 0} : V \rightarrow \C$ from (\ref{canonform}). First we recall the definition of the $g$-twisted Zhu algebra $\zhu_g(V)$.

\vspace{.3cm}
\begin{defn} \label{zhualgdefn}
For $u, v \in V$, $n \in \Z$ let
\begin{align*}
u \circ_n v &= \res_w w^n (1+w)^{\D_u + \eps_u} Y(u, w)v dw \\
&= 2\pi i \res_z e^{2\pi i (1+\eps_u) z} (e^{2\pi i z}-1)^n Y[u, z]v dz.
\end{align*}
Let $J_g \subseteq V$ be the span of all elements of the form
\begin{align*}
&u \quad \text{for $\mu(u) \neq 1$} \\
&u \circ_{n} v \quad \text{for $\mu(u) = \mu(v) = 1$ and $n \leq -2$}, \\
&u \circ_{n} v \quad \text{for $\mu(u) = \mu(v)^{-1} \neq 1$ and $n \leq -1$}, \\
\text{and} \quad &(L_{-1} + L_0)u \quad \text{for $\mu(u) = 1$}.
\end{align*}
The \emph{$g$-twisted Zhu algebra} is $\zhu_g(V) = V / J_g$ as a vector superspace, with the product induced by $\circ_{-1}$.
\end{defn}

We denote the projection of $u \in V$ to $\zhu_g(V)$ as $[u]$ or simply $u$. The following two theorems are proved in \cite{DK}; note that $u_{([0])}v$ in our notation is $(2\pi i)^{-1} [u, v]_{\hbar=1}$ in theirs.

\vspace{.3cm}
\begin{thm} \label{zhudefn} {\ }
\begin{itemize}
\item The product $\circ_{-1}$ is well-defined on $\zhu_g(V)$ and makes it into an associative superalgebra with unit $[\vac]$. We denote the product by $*$.

\item The $0^\text{th}$ Zhu product $\cdot_{([0])}\cdot$ is well-defined on $\zhu_g(V)$ and we have
\begin{align} \label{skewsymm}
u * v - p(u, v) v * u = 2\pi i u_{([0])}v \quad \text{for all $u, v \in \zhu_g(V)$}.
\end{align}

\item $[\om]$ is central in $\zhu_g(V)$.
\end{itemize}
\end{thm}

\vspace{.3cm}
\begin{thm} \label{functorial} {\ }
\begin{itemize}
\item There is a restriction functor $\Omega$ from the category of positive energy $g$-twisted $V$-modules to the category of $\zhu_g(V)$-modules. It sends $M$ to $M_0$ with the action $[u] * x = u^M_0 x$ for $u \in V$ and $x \in M_0$.

\item There is an induction functor $L$ going in the other direction, and we have $\Omega(L(N)) \cong N$ for any $\zhu_g(V)$-module $N$.

\item $\Omega$ and $L$ are inverse bijections between the sets of irreducible modules in each category.
\end{itemize}
\end{thm}

The automorphism $h$ of $V$ descends to an automorphism of $\zhu_g(V)$, which we also denote $h$.

\vspace{.3cm}
\begin{prop} \label{cpistrace}
Let $S \in \CC(g, h)$ with Frobenius expansion (\ref{canonform}). Fix $j \in \{1, 2, \ldots b(p)\}$, and let $f = C_{p, j, 0}$. We have
\begin{itemize}
\item $f(u) = 0$ for all $u \in J_g(V)$, so $f$ descends to a map $f : \zhu_g(V) \rightarrow \C$.

\item $f(u * v) =  \delta_{\la(u)\la(v), 1} p(u, v) \la(u)^{-1} f(v * u)$ for all $u, v \in \zhu_g(V)$.
\end{itemize}
\end{prop}

\begin{proof}
By definition $S(\cdot, \tau)$ annihilates $\OO(g, h)$. Therefore $f$ annihilates the $q^0$ coefficient of any element of $\OO(g, h)$. If $u \in V$ with $\mu(u) \neq 1$ then $u \in \OO(g, h)$, so $f(u) = 0$. If $\mu(u) = 1$ then $2\pi i (L_{-1} + L_0)u = \tilde{\omega}_{([0])}u \in \OO(g, h)$ is annihilated by $f$. Now
\begin{align*}
[q^0] : \ide_1(u, v) &= \ide_1(u, v) = u_{([0])}v, \\
%%%%%
[q^0]: \ide_2(u, v)
&= \res_z \left([q^0] : \partial_z P(z, q) - G_2(q)\right) Y[u, z]v dz \\
&= \res_z \left[ 2\pi i \partial_z \frac{e^{2\pi i z}}{e^{2\pi i z}-1} - 2\zeta(2) \right] Y[u, z]v dz \\
&= - (2\pi i)^2 \res_z \frac{e^{2\pi i z}}{(e^{2\pi i z}-1)^2} Y[u, z]v dz - 2\zeta(2) u_{([0])}v \\
&= - 2\pi i u \circ_{-2} v - 2 \zeta(2) u_{([0])}v, \\
%%%%%%
\text{and} \quad [q^0]: \ide_3(u, v)
&= \res_z \left( [q^0] : P^{\mu(u), \la(u)}(z, q) \right) Y[u, z]v dz \\
&= \res_z \left[\frac{2\pi i \delta}{1-\la(u)} - 2\pi i \frac{e^{2\pi i(1+\eps_u)z}}{e^{2\pi i z}-1} \right] Y[u, z]v dz \\
&= \frac{2\pi i \delta}{1-\la(u)} u_{([0])}v - u \circ_{-1} v.
\end{align*}
Hence $f$ annihilates $J_g$, and descends to a function on $\zhu_g(V)$.

Now for the second part. Let $\mu(u) = \mu(v) = 1$. If $\la(u) \la(v) \neq 1$ then $u * v$, $v * u$ and $u_{([0])}v$ all lie in $\OO(g, h)$ and so are annihilated by $f$. If $\la(u) = \la(v) = 1$ then $f$ annihilates $u_{([0])}v$, hence $f(u * v) = p(u, v) f(v * u)$. If $\la(u) \la(v) = 1$ with $\la(u) \neq 1$, then $f$ annihilates
\[
\frac{2\pi i}{1-\la(u)} u_{([0])}v - u * v.
\]
Combining this with (\ref{skewsymm}) shows that
\[
f(u * v) = p(u, v) \la(u)^{-1} f(v * u),
\]
so we are done.
\end{proof}

%%%%%%%%%%%%%%%%%%%%%%%%%%%%%%%%%%%%%%%%%%%%%%%%%%%%%%%%%%%%%%%%%%%%%%%%%%%%%%%%%%%%%%%%
%%%%%%%%%%%%%%%%%%%%%%%%%%%%%%%%%%%%%%%%%%%%%%%%%%%%%%%%%%%%%%%%%%%%%%%%%%%%%%%%%%%%%%%%
%%%%%%%%%%%%%%%%%%%%%%%%%%%%%%%%%%%%%%%%%%%%%%%%%%%%%%%%%%%%%%%%%%%%%%%%%%%%%%%%%%%%%%%%
%%%%%%%%%%%%%%%%%%%%%%%%%%%%%%%%%%%%%%%%%%%%%%%%%%%%%%%%%%%%%%%%%%%%%%%%%%%%%%%%%%%%%%%%
%%%%%%%%%%%%%%%%%%%%%%%%%%%%%%%%%%%%%%%%%%%%%%%%%%%%%%%%%%%%%%%%%%%%%%%%%%%%%%%%%%%%%%%%
%%%%%%%%%%%%%%%%%%%%%%%%%%%%%%%%%%%%%%%%%%%%%%%%%%%%%%%%%%%%%%%%%%%%%%%%%%%%%%%%%%%%%%%%

\section{$h$-supersymmetric functions} \label{symmfunc}

Let $A$ be a finite dimensional unital associative superalgebra carrying an automorphism $h$ of finite order. Let $\la(a)$ denote the $h$-eigenvalue of an eigenvector $a \in A$. We consider linear functions $f : A \rightarrow \C$ satisfying
\begin{align} \label{nearsusy}
f(a b) = \delta_{\la(a)\la(b), 1} p(a, b) \la(a)^{-1} f(b a)
\end{align}
for all $a, b \in A$ eigenvectors of $h$. We refer to these functions as \emph{$h$-supersymmetric} functions on $A$, and write $\mathcal{F}_h(A)$ for the space of all such functions. We also write $\mathcal{F}(A)$ for $\mathcal{F}_1(A)$ and refer to these as \emph{supersymmetric} functions.

Let $A$ be semisimple now, and let $A = \oplus_{i \in I} A_i$ be its decomposition into simple components. The automorphism $h$ permutes the $A_i$. The following fact is proved in Lemma 10.7 of \cite{DLMorbifold} (for algebras, but the proof carries over naturally to the $\Z/2\Z$-graded setting).
\vspace{.3cm}
\begin{lemma}
$\mathcal{F}_h(A) = \oplus_{i \in J} \mathcal{F}_h(A_i)$ where the direct sum is over the subset $J \subset I$ of $h$-invariant simple components.
\end{lemma}
For the rest of this section let $A$ be a simple superalgebra, i.e., having no proper nonzero $\Z/2\Z$-graded ideals. Note that a simple superalgebra need not be simple as an algebra. We summarize some well known (see, e.g., \cite{Kacsuper}) results about simple superalgebras in the following theorem.
\vspace{.3cm}
\begin{thm}\label{superstructure}
Let $A$ be a finite dimensional simple superalgebra over $\C$. Then
\begin{itemize}
\item (Superalgebra Wedderburn theorem) $A$ is isomorphic either to $\en(\C^{m|k})$ or else to $Q_n = \en(\C^n)[\suv] / (\suv^2 = 1)$ where $\suv$ is an odd indeterminate. Henceforth we refer to these cases as {\sf Type I} and {\sf Type II}, respectively.

\item $A$ has a unique irreducible module up to isomorphism. It is $\C^{m|k}$ in the {\sf Type I} case, and $\C^{n|n} = \C^n + \C^n \suv$ in the {\sf Type II} case, with the obvious action in each case.

\item (Superalgebra Schur lemma) Let $N$ denote the unique up to isomorphism irreducible $A$-module. Then
\begin{align*}
\en_A(N) = \left\{\begin{array}{ll}
\C 1 & \text{if $A$ is {\sf Type I}}, \\
\C 1 + \C \suv & \text{if $A$ is {\sf Type II}}. \\
\end{array}\right.
\end{align*}
\end{itemize}
\end{thm}

\begin{proof}
Lemma 3 of \cite{Wall} states that $A$ is either simple as an ordinary algebra or else it is isomorphic to $Q_n$ for some $n \geq 1$. In the first case $A = \en(N)$ for some finite dimensional vector space $N$ which is in turn the unique irreducible $A$-module. Consider the involution $\sigma_A$ on $A$. There exists an invertible element (indeed an involution) $b \in A$ such that $\sigma(a) = b^{-1} a b$ for all $a \in A$. Now $N$ decomposes into eigenspaces for $b$ with eigenvalues $1$ and $-1$. We make $N$ into a superspace by declaring these eigenspaces even and odd respectively. Then $b$ is just $\sigma_N$. Now $A \cong \en(\C^{m|k})$ as a superalgebra and $N$ is its unique irreducible module. In the second case any irreducible $A$-module $N$ is in particular an $\en(\C^n)$-module and therefore a direct sum of copies of $\C^n$. Since $\suv$ maps an even copy of $\C^n$ to an odd one, and $\suv^2 = 1$, we deduce that $A$ has the unique irreducible $\Z/2\Z$-graded module $N = \C^n + \C^n \suv$ (observe that without the $\Z/2\Z$-graded condition $N$ would not be unique). This completes the proofs of the first two items.

For the third item: the {\sf Type I} case follows from the usual Schur lemma by forgetting the $\Z/2\Z$-grading. If $A$ is of {\sf Type II} let $\phi \in \en_A(N)_0$, then $\phi|_{N_0} \in \C 1$ by the usual Schur lemma and since $\phi$ commutes with $\suv$ we have $\phi \in \C 1$. If $\phi \in \en_A(N)_1$ then clearly $\phi \suv \in \en_A(N)_0$ and so $\phi \in \C \suv$.
\end{proof}
Supersymmetric functions on simple superalgebras were first considered in \cite{KacAlgebraic} and are characterized as follows.
\vspace{.3cm}
\begin{lemma} \label{whataresymm}
Let $A$ be a finite dimensional simple superalgebra, and $N$ its unique up to isomorphism irreducible module. The space $\mathcal{F}(A)$ of supersymmetric functions on $A$ is one dimensional and is spanned by the function
\begin{align*}
\begin{array}{ll}
a \mapsto \str_N(a) & \text{if $A$ is of {\sf Type I}}, \\
a \mapsto \tr_N(a \suv) & \text{if $A$ is of {\sf Type II}}. \\
\end{array}
\end{align*}
\end{lemma}

\begin{proof}
Let $f$ be a supersymmetric function on $A = \en(\C^{m|k})$. Fix a $\Z/2\Z$-homogeneous basis $e_1, \ldots e_{m+k}$ of $\C^{m|k}$ where $e_1, \ldots, e_m$ are even and $e_{m+1}, \ldots, e_{m+k}$ odd. Define $E_{ij} \in A$ by $E_{ij}(e_k) = \delta_{ik} e_j$. If $i \neq j$ then
\[
f(E_{ij}) = f(E_{ii} E_{ij}) = \pm f(E_{ij} E_{ii}) = \pm f(0) = 0.
\]
Thus $f(a) = \sum_i k_i a_{ii}$ is a linear combination of the diagonal entries of $a$. We now have
\[
f(E_{ii}) = f(E_{ij} E_{ji}) = \pm f(E_{ji} E_{ij}) = \pm f(E_{jj}),
\]
where the sign $\pm$ is a $+$ if $E_{ij}$ is even (i.e., if $e_i$ and $e_j$ have the same parity) and is a $-$ if $E_{ij}$ is odd (i.e., if $e_i$ and $e_j$ have opposite parity). Therefore, up to a scalar multiple $k_i = 1$ for $i = 1, \ldots, m$ and $k_i = -1$ for $i = m+1, \ldots, m+k$. Hence $f$ is proportional to $\str_N$ where $N = \C^{m|k}$ is the unique irreducible $A$-module.

Now let $f$ be a supersymmetric function on $A = Q_n$. For $a \in A_0$ we have $f(a) = f(\suv(\suv a)) = -f((\suv a)\suv) = -f(a)$, hence $f(a) = 0$. On the other hand $g : A_0 \rightarrow \C$ defined by $g(a) = f(a \suv)$ is symmetric because $g(ab) = f(ab \suv) = f(b \suv a) = f(ba \suv) = g(ba)$. Therefore $g$ is a scalar multiple of $\tr_{N_0}$ and $f$ is a scalar multiple of the map $a \mapsto \tr_{N_0}(a \suv) = (1/2)\tr_N(a \suv)$.
\end{proof}
Now we turn to the description of $h$-supersymmetric functions. First we need the following superalgebra analog of the Skolem-Noether theorem.
\vspace{.3cm}
\begin{lemma}\label{choosingiota}
Let $A$ be a finite dimensional simple superalgebra, $N$ its irreducible module, and let $h$ be an automorphism of $A$. There exists invertible $\iota \in \en_\C(N)$ such that $h(a) x = \iota^{-1} a \iota x$ for all $a \in A$, $x \in N$. If $A$ is of {\sf Type I} then $\iota$ is $\Z / 2\Z$-homogeneous and is unique up to a nonzero scalar factor. If $A$ is of {\sf Type II} then $\iota$ (while not unique) may be chosen to be $\Z/2\Z$-homogeneous of either parity.
\end{lemma}

\begin{proof}
The map $(a, x) \mapsto h(a)x$ defines an irreducible action of $A$ on $N$ which we denote $h \cdot N$ to distinguish it from the original action. By the second item of Theorem \ref{superstructure} the $A$-modules $N$ and $h \cdot N$ are equivalent, i.e., there exists invertible $\iota \in \Hom_A(h \cdot N, N)$. A choice of such $\iota$ identifies $\Hom_A(h \cdot N, N)$ with $\en_A(N)$.

Let $A$ be of {\sf Type I}. Then $\en_A(N) = \C 1$ and so $\iota$ is unique up to a nonzero scalar factor. By definition $h$ commutes with the parity operator $\sigma_A$. This implies $(\sigma_N\iota)^{-1}a(\sigma_N\iota) = (\iota \sigma_N)^{-1}a(\iota \sigma_N)$ for all $a \in A$. Hence by the Schur lemma $\sigma_N \iota = \eps \iota \sigma_N$ for some nonzero constant $\eps$. Since the eigenvalues of $\sigma_N$ are $\pm 1$ we have $\eps = \pm 1$ and in either case $\iota$ is $\Z/2\Z$-homogeneous.

Let $A$ be of {\sf Type II}. There exists $\iota_0 \in A_0$ such that $h(a) = \iota_0^{-1} a \iota_0$ for all $a \in A_0$. Since we have $h(\suv)^2 = h(\suv^2) = h(1) = 1$ and $h(a)h(\suv) = h(a\suv) = h(\suv a) = h(\suv)h(a)$ for all $a \in A$ we deduce $h(\suv) = \pm \suv$. Write $h(\suv) = (-1)^p \suv$. By the Schur lemma $\Hom_A(h \cdot N, N)$ is two dimensional. One then easily verifies $\Hom_A(h \cdot N, N) = \C \iota_0 \sigma_N^p + \C \iota_0\sigma_N^p\suv$. The claim follows.
\end{proof}
To facilitate the description of $h$-supersymmetric functions we make the following definition.
\vspace{.3cm}
\begin{defn}\label{choosinggamma}
Let $A$ be a finite dimensional simple superalgebra, $N$ its irreducible module, and let $h$ be an automorphism of $A$. We define $\ga$ to be an invertible element of $\en_\C(N)$ satisfying $h(a) x = \ga^{-1} a \ga x$ for all $a \in A$, $x \in N$. If $A$ is of {\sf Type II} we further require that $\ga$ be even (resp. odd) if $h(\suv) = \suv$ (resp. $h(\suv) = -\suv$).
\end{defn}
The map $\ga$ so defined is unique up to a nonzero scalar factor.
\vspace{.3cm}
\begin{lemma} \label{htoordlemma}
Let $A$ be a finite dimensional simple superalgebra, $N$ its irreducible module, and let $h$ be an automorphism of $A$ of finite order. The space $\mathcal{F}_h(A)$ of $h$-supersymmetric functions on $A$ is one dimensional and is spanned by the function
\begin{align*}
\begin{array}{ll}
a \mapsto \str_N \left(a \ga \sigma_N^{p(\ga)}\right) & \text{if $A$ is of {\sf Type I}}, \\
a \mapsto \tr_N \left(a \ga \sigma_N^{p(\ga)} \suv\right) & \text{if $A$ is of {\sf Type II}}. \\
\end{array}
\end{align*}
\end{lemma}
We shall sometimes abuse notation slightly by writing both cases as $a \mapsto F(a \ga \sigma_N^{p(\ga)})$ where $F \in \mathcal{F}(A)$.
\begin{proof}
Let $A$ be of {\sf Type I}. Since $h(\ga) = \ga^{-1} \ga \ga = \ga$ we have $\la(\ga) = 1$. Also $h(\sigma_N) = \ga^{-1} \sigma_N \ga = (-1)^{p(\ga)} \sigma_N$ so $\la(\sigma_N) = (-1)^{p(\ga)}$. Let $f \in \mathcal{F}_h(A)$ and put $\ov{f}(a) = f(a \ga^{-1})$ if $\ga$ is even and $\ov{f}(a) = f(a \ga^{-1} \sigma_N)$ if $\ga$ is odd. Then if $\ga$ is even we have
\begin{align} \label{htoord}
\begin{split}
\ov{f}(ab) = f(ab\ga^{-1})
&= \la(a)^{-1} \delta_{\la(a)\la(b\ga^{-1}), 1} p(a, b\ga^{-1}) f(b \ga^{-1} a) \\
&= \la(a)^{-1} \delta_{\la(a)\la(b), 1} p(a, b) f(b \ga^{-1} a \ga \ga^{-1}) \\
&= \delta_{\la(a)\la(b), 1} p(a, b) f(b a \ga^{-1}) \\
&= \delta_{\la(a)\la(b), 1} p(a, b) \ov{f}(b a),
\end{split}
\end{align}
while if $\ga$ is odd we have
\begin{align} \label{htoord2}
\begin{split}
\ov{f}(ab) = f(ab\ga^{-1}\sigma_N)
&= \la(a)^{-1} \delta_{\la(a)\la(b\ga^{-1}\sigma_N), 1} p(a, b\ga^{-1}\sigma_N) f(b \ga^{-1} \sigma_N a) \\
&= \la(a)^{-1} \delta_{\la(a)\la(b), -1} p(a, b\ga^{-1}) (-1)^{p(a)} f(b \ga^{-1} a \sigma_N) \\
&= \la(a)^{-1} \delta_{\la(a)\la(b), -1} p(a, b) f(b \ga^{-1} a \ga \ga^{-1} \sigma_N) \\
&= \delta_{\la(a)\la(b), -1} p(a, b) f(b a \ga^{-1} \sigma_N) \\
&= \delta_{\la(a)\la(b), -1} p(a, b) \ov{f}(b a).
\end{split}
\end{align}
In particular $\ov{f}$ is supersymmetric in both cases and therefore is a multiple of $\str_N$. Therefore $f$ is a scalar multiple of $a \mapsto \str_N(a\ga \sigma_N^{p(\ga)})$. It remains to verify that the latter function is $h$-supersymmetric; for this we need only show that $\str_N(a\ga \sigma_N^{p(\ga)}) \neq 0$ only when $\la(a) = 1$. Well, $N$ splits into eigenspaces for $\ga$ and $a$ maps the $\eps$ eigenspace to the $\la(a) \eps$ eigenspace. In general $\str_N(b) = \tr_N(b\sigma_N)$ vanishes unless $\la(a\sigma_N) = 1$, so $\str_N(a\ga \sigma_N^{p(\ga)})$ vanishes unless $1 = \la(a) \la(\ga) \la(\sigma_N)^{p(\ga)+1} = \la(a)$.

Now let $A$ be of {\sf Type II} and assume first that $h(\suv) = \suv$. For $f \in \mathcal{F}_h(A)$ put $\ov{f}(a) = f(a \ga^{-1})$. We may repeat calculation (\ref{htoord}) to deduce that $\ov{f}$ is a supersymmetric function on $A$, hence a multiple of $a \mapsto \tr_N(a \suv)$. Therefore $f$ is a scalar multiple of $a \mapsto \tr_N(a \ga \suv)$.

The final case of $h(\suv) = -\suv$ is a little more delicate because $\ga = \iota_0 \sigma_N \suv$ (where $\iota_0$ is as in the proof of Lemma \ref{choosingiota}) is not an element of $A$. For $f \in \mathcal{F}_h(A)$ we let $\ov{f}(a) = f(a\iota_0^{-1})$ and repeat calculation (\ref{htoord}). Since $\iota_0^{-1} a \iota_0 = (-1)^{p(a)} h(a)$ we obtain this time
\[
\ov{f}(ab) = \delta_{\la(a)\la(b), 1} p(a, b) (-1)^{p(a)} \ov{f}(ba)
\]
for all $a, b \in A$. For $a \in A_0$ we obtain
\[
\ov{f}(\suv a) = -\delta_{\la(a), -1} \ov{f}(a\suv),
\]
hence $\ov{f}(A_1) = 0$. For $a, b \in A_0$ we obtain
\[
\ov{f}(ab) = \delta_{\la(a)\la(b), 1} \ov{f}(ba),
\]
hence $\ov{f}$ is a symmetric function on $A_0$. It follows that $\ov{f}$ is a scalar multiple of $\tr_{N_0} = (1/2)\tr_N$. This means that $f$ is a scalar multiple of
\[
a \mapsto \tr_{N}(a \iota_0) = \tr_N(a \ga \suv \sigma_N) = -\tr_N(a \ga \sigma_N \suv).
\]

Finally we should verify that $\tr_N(a \ga \sigma_N^{p(\ga)} \suv) \neq 0$ only if $\la(a) = 1$. In the case $h(\suv) = \suv$ the trace vanishes unless $1 = \la(a) \la(\ga) \la(\suv) = \la(a)$. In the case $h(\suv) = -\suv$ the trace reduces to $\tr_{N_0}(a\iota_0)$ which again vanishes unless $\la(a) = 1$.
\end{proof}

%%%%%%%%%%%%%%%%%%%%%%%%%%%%%%%%%%%%%%%%%%%%%%%%%%%%%%%%%%%%%%%%%%%%%%%%%%%%%%%%%%%%%%%%
%%%%%%%%%%%%%%%%%%%%%%%%%%%%%%%%%%%%%%%%%%%%%%%%%%%%%%%%%%%%%%%%%%%%%%%%%%%%%%%%%%%%%%%%
%%%%%%%%%%%%%%%%%%%%%%%%%%%%%%%%%%%%%%%%%%%%%%%%%%%%%%%%%%%%%%%%%%%%%%%%%%%%%%%%%%%%%%%%

\section{Supertrace Functions} \label{tracefunc}

In this section we use the results of Section \ref{symmfunc} to associate a supertrace function to each pair $(M, h)$ where $M$ is an irreducible positive energy $g$-twisted $V$-module, and $h$ is an automorphism of $V$ of finite order commuting with $g$. We show that this function lies in $\CC(g, h)$.

The Zhu algebra $\zhu_g(V)$ is finite dimensional (because $V$ is $C_2$-cofinite) and is semisimple by assumption. Let $A \subseteq \zhu_g(V)$ be a $h$-invariant simple component, $N$ its irreducible module, and let $\ga : N \rightarrow N$ be as in Definition \ref{choosinggamma}. It follows from Theorem \ref{functorial} that $\ga$ lifts to a grade-preserving map (which we also denote $\ga$) of the associated $V$-module $M = L(N)$ satisfying
\begin{align} \label{intertwine}
\ga^{-1} u^M_n \ga = h(u)^M_n
\end{align}
for all $u \in V$, $n \in [\eps_u]$. In the case that $A$ is of {\sf Type II} we shall make use of the involution $\suv$ of $M = L(N)$ lifted from $\suv : N \rightarrow N$ in the same way, i.e.,
\[
\suv u^M_n \suv = (-1)^{p(u)} u^M_n.
\]

Below we shall drop the $M$ superscripts. We sometimes refer to $M$ as being of {\sf Type I} (resp. {\sf II}) if its associated simple component $\zhu_g(V)$ is of {\sf Type I} (resp {\sf II}).
\vspace{.3cm}
\begin{defn}
Let $M$ be a positive energy $g$-twisted $V$-module and $h$ an automorphism of $V$ of finite order commuting with $g$. For $\alpha : M_r \rightarrow M_r$ an endomorphism of a graded piece of $M$ we define
\begin{align*}
T_{M_r}(\alpha) = \left\{\begin{array}{ll}
\str_{M_r}\left(\alpha \ga \sigma_M^{p(\ga)}\right) & \text{if $M$ is of {\sf Type I}}, \\
\tr_{M_r}\left(\alpha \ga \sigma_M^{p(\ga)} \suv\right) & \text{if $M$ is of {\sf Type II}}. \\
\end{array} \right.
\end{align*}
We define $T_M(\al) = \sum_r T_{M_{r}}(\al)$ whenever the sum is well-defined. Finally we define the \emph{supertrace function} associated to $M$ and $h$ to be
\begin{align}\label{supertracefunctiondefinition}
S_M(u, \tau) = T_M(u_0 q^{L_0 - \mathfrak{c}/24}).
\end{align}
\end{defn}
The main theorem of this section is
\vspace{.3cm}
\begin{thm} \label{tf-in-cb}
Let $M$ be a $h$-invariant irreducible positive energy $g$-twisted $V$-module. Then the supertrace function $S_M$ defined by (\ref{supertracefunctiondefinition}) lies in $\CC(g, h)$. Furthermore the $S_M$, as $M$ ranges over all such modules, are linearly independent.
\end{thm}

\begin{proof}
The $S_M$ are linearly independent because the $h$-supersymmetric functions on $\zhu_g(V)$ are (each is supported on a different simple component). The proof that $S_M \in \CC(g, h)$ is carried out in Propositions \ref{mula=0}-\ref{trtodef2} below.

To summarize: in Propositions \ref{mula=0} and \ref{a0b} we show that $T_M$ annihilates all $u \in V$ satisfying $(\mu(u), \la(u)) \neq (1, 1)$, and annihilates $u_{([0])}v$ for $(\mu(u), \la(u)) = (1, 1)$. In Proposition \ref{a-1b} we establish an identity which is used in Proposition \ref{trtodef1} to show that $T_M$ annihilates the remaining elements of $\OO(g, h)$. These properties pass immediately to $S_M$ and thus {\bf CB3} is verified.

In Proposition \ref{trtodef2} we use Proposition \ref{a-1b} again to show that $S_M$ satisfies {\bf CB4}.

Axiom {\bf CB2} is automatic. 

To verify {\bf CB1}, we must show that the power series $S_M$ converges to a holomorphic function in $|q| < 1$. In the presence of the $C_2$-cofiniteness condition {\bf CB1} follows from the other axioms. This is because the calculations of Section \ref{ODEsection} show the power series $S_M$ formally satisfies a Fuchsian ODE and it follows that $S_M$ converges to a solution of this ODE (see \cite{DLMorbifold}).
\end{proof}

Since $\om \in V$ is $h$-invariant, $\ga$ commutes with $L_0$ and $q^{L_0}$. Since $\om$ is in the center of $\zhu_g(V)$ and its module $N = M_0$ is irreducible, $L_0$ acts on $N$ as a scalar (called the \emph{conformal weight} of $M$ and often denoted $h(M)$). Recall that $[L_0, u_k] = -ku_k$. It follows that $L_0$ acts on $M_r$ as the scalar $h(M) + r$, and that $u_0$ commutes with $L_0$ and $q^{L_0}$.

\vspace{.3cm}
\begin{prop} \label{mula=0}
We have
\begin{enumerate}[(a)]
\item If $(\mu(u), \la(u)) \neq (1, 1)$, then $T_{M_r}(u_0) = 0$.

\item For all $u, v \in V$ we have $T_{M_r}(u_0v_0) = \delta_{\la(u)\la(v), 1} p(u, v) \la(u)^{-1} T_{M_r}(v_0u_0)$.

\item If $\mu(v) = \mu(u)^{-1}$ and $n \in [\eps_u]_{>0}$ then
\[
T_{M_r} u_n v_{-n} = \la(u)^{-1} p(u, v) T_{M_{r+n}} v_{-n} u_{n}.
\]
\end{enumerate}
\end{prop}

\begin{proof}
We begin with (b). We observe that the work has already been done in the proof of Lemma \ref{htoordlemma}. Indeed $T_{M_0}$ is simply the $h$-supersymmetric function introduced in that lemma. For $r > 0$ we define the superalgebra $A^r$ to be $A^r = \en_{\C}(M_r)$ if $M$ is of {\sf Type I}, and $A^r = A^r_0 + \suv A^r_0$ where $A^r_0 = \en_{\C}((M_r)_0)$ if $M$ is of {\sf Type II}. Now the proof of Lemma \ref{htoordlemma} can be repeated, with the pair $(A, N)$ replaced by $(A^r, M_r)$, to establish (b).

In (c) we may replace $M_r$ and $M_{r+n}$ by their direct sum because $v_{-n}$ annihilates $M_{r+n}$ and $u_n$ annihilates $M_r$. The equation then follows from the same argument as for (b).

For (a): the operators $u_n$ are defined for $n \in \eps_u + \Z$, so if $\mu(u) \neq 1$ then $u_0 = 0$ and $T_{M_r}(u_0) = 0$. Once again the proof that $T_{M_r}(u_0) = 0$ only if $\la(u) = 1$ has been carried out for $r = 0$ in the proof of Lemma \ref{htoordlemma}, and the general case is the same.
\end{proof}

%%%%%%%%%%%%%%%%%%%%%%%%%%%%%%%%%%%%%%%%%%%%%%%%%%%%%%%%%%%%%%%%%%%%%%%%%%%%%%%%%%%%%%%%%%%%%%
%%%%%%%%%%%%%%%%%%%%%%%%%%%%%%%%%%%%%%%%%%%%%%%%%%%%%%%%%%%%%%%%%%%%%%%%%%%%%%%%%%%%%%%%%%%%%%
%%%%%%%%%%%%%%%%%%%%%%%%%%%%%%%%%%%%%%%%%%%%%%%%%%%%%%%%%%%%%%%%%%%%%%%%%%%%%%%%%%%%%%%%%%%%%%
%%%%%%%%%%%%%%%%%%%%%%%%%%%%%%%%%%%%%%%%%%%%%%%%%%%%%%%%%%%%%%%%%%%%%%%%%%%%%%%%%%%%%%%%%%%%%%
%%%%%%%%%%%%%%%%%%%%%%%%%%%%%%%%%%%%%%%%%%%%%%%%%%%%%%%%%%%%%%%%%%%%%%%%%%%%%%%%%%%%%%%%%%%%%%

\vspace{.3cm}
\begin{prop} \label{a0b}
If $\mu(u) = 1$ then
\[
T_{M_r}(u_{([0])}v)_0 = [1 - \la(u)] (2\pi i)^{-1} T_{M_r}u_0 v_0.
\]
\end{prop}

\begin{proof}
Assume $\mu(v) = 1$ and $\la(v) = \la(u)^{-1}$, for otherwise both sides of the equation vanish and the result is trivially true.

The commutator formula (\ref{commutator}) with $m = k = 0$ is
\begin{align*}
[u_0, v_0]
= \sum_{j \in \Z_+} \binom{\D_u-1}{j} (u_{(j)}v)_0
= (\res_w (1+w)^{\D_u-1} Y(u, w)v dw)_0.
\end{align*}
Using the substitution $w = e^{2\pi i z} - 1$ gives
\begin{align*}
(u_{([0])}v)_0
&= (\res_z Y[u, z]v dz)_0
= \left( \res_z e^{2\pi i \D_u z} Y(u, e^{2\pi i z}-1)v dz \right)_0 \\
&= (2\pi i)^{-1} \left( \res_w (1+w)^{\D_u-1} Y(u, w)v dw \right)_0 \\
&= (2\pi i)^{-1} [u_0, v_0].
\end{align*}
Now we use Proposition \ref{mula=0} (b) to simplify
\begin{align*}
T_{M_r} [u_0, v_0]
&= T_{M_r} u_0 v_0 - p(u, v) T_{M_r} v_0 u_0 \\
&= T_{M_r} u_0 v_0 - \delta_{\la(v)\la(u), 1} p(u, v) \la(v)^{-1} T_{M_r} v_0 u_0 \\
&= [1 - \la(u)] T_{M_r} u_0 v_0.
\end{align*}
\end{proof}

%%%%%%%%%%%%%%%%%%%%%%%%%%%%%%%%%%%%%%%%%%%%%%%%%%%%%%%%%%%%%%%%%%%%%%%%%%%%%%%%%%%%%%%%%%%%%%
%%%%%%%%%%%%%%%%%%%%%%%%%%%%%%%%%%%%%%%%%%%%%%%%%%%%%%%%%%%%%%%%%%%%%%%%%%%%%%%%%%%%%%%%%%%%%%
%%%%%%%%%%%%%%%%%%%%%%%%%%%%%%%%%%%%%%%%%%%%%%%%%%%%%%%%%%%%%%%%%%%%%%%%%%%%%%%%%%%%%%%%%%%%%%
%%%%%%%%%%%%%%%%%%%%%%%%%%%%%%%%%%%%%%%%%%%%%%%%%%%%%%%%%%%%%%%%%%%%%%%%%%%%%%%%%%%%%%%%%%%%%%
%%%%%%%%%%%%%%%%%%%%%%%%%%%%%%%%%%%%%%%%%%%%%%%%%%%%%%%%%%%%%%%%%%%%%%%%%%%%%%%%%%%%%%%%%%%%%%

\vspace{.3cm}
\begin{prop} \label{a-1b}
\begin{align} \label{pro24eq}
T_M((\res_z P^{\mu(u), \la(u)}(z, q) Y[u, z]v dz)_0 q^{L_0}) =
\left\{
\begin{array}{ll}
-T_M(u_0 v_0 q^{L_0}) & \text{if $(\mu(u), \la(u)) = (1, 1)$}, \\
0 & \text{otherwise}. \\
\end{array}
\right.
\end{align}
\end{prop}

\begin{proof}
Let $\eps = \eps_u$, $\mu = \mu(u)$ and $\la = \la(u)$. Assume that $\mu(v) = \mu^{-1}$ and $\la(v) = \la^{-1}$, for otherwise both sides of the claimed equality vanish automatically.

For any $r \geq 0$, the $q^{h(M) + r}$ coefficient of
\[
T_M (\res_z P^{\mu, \la}(z, q) Y[u, z]v dz)_0 q^{L_0}
\]
is $X - Y + Z$, where
\begin{align} \label{a-1bvar}
\begin{split}
X &= \frac{2\pi i \delta}{1-\la} T_{M_r} (u_{([0])}v)_0, \\
Y &= 2\pi i T_{M_r} (\res_z \frac{e^{2\pi i (1+\eps) z}}{e^{2\pi i z}-1} Y[u, z]v dz)_0, \\
\text{and} \quad Z &= 2\pi i \sum_{m \in \Z_{>0}} \la^m \sum_{n \in [\eps]_{>0}} T_{M_{r-mn}} (\res_z e^{2\pi i nz} Y[u, z]v dz)_0 \\
&- 2\pi i \sum_{m \in \Z_{>0}} \la^{-m} \sum_{n \in [\eps]_{<0}} T_{M_{r+mn}} (\res_z e^{2\pi i nz} Y[u, z]v dz)_0
\end{split}
\end{align}
(the sum defining $Z$ is finite since terms with $|mn| > r$ contribute nothing).

Using the change of variable $w = e^{2\pi i z} - 1$ and the commutator formula, we have
\begin{align} \label{Pnreduction}
\begin{split}
\left( \res_z e^{2\pi i n z} Y[u, z]v dz \right)_0
&= (2\pi i)^{-1} \left( \res_w (1+w)^{n+\D_u-1} Y(u, w)v dw \right)_0 \\
&= (2\pi i)^{-1} \sum_{j \in \Z_+} \binom{n+\D_u-1}{j} (u_{(j)}v)_0
= (2\pi i)^{-1} [u_n, v_{-n}]
\end{split}
\end{align}
for any $n \in [\eps_u]$. Now let $n \in [\eps_u]_{>0}$. Proposition \ref{mula=0} (c) implies that
\begin{align*}
\sum_{m \in \Z_{>0}} \la^m T_{M_{r-mn}} u_n v_{-n}
%%%%
&= p(u, v) \la^{-1} \sum_{m \in \Z_{>0}} \la^m T_{M_{r-(m-1)n}} v_{-n} u_{n} \\
%%%%
&= p(u, v) \sum_{m \in \Z_+} \la^{m} T_{M_{r-mn}} v_{-n} u_{n},
\end{align*}
from which it follows that
\begin{align} \label{tel}
\sum_{m \in \Z_{>0}} \la^m T_{M_{r-mn}} [u_n, v_{-n}] = p(u, v) T_{M_r} v_{-n} u_{n}.
\end{align}
If $n \in [\eps_u]_{<0}$ then $-n \in [\eps_v]_{>0}$, and so
\begin{align} \label{tel2}
\begin{split}
\sum_{m \in \Z_{>0}} \la^{-m} T_{M_{r+mn}} [u_n, v_{-n}]
&= -p(u, v) \sum_{m \in \Z_{>0}} \la^{-m} T_{M_{r-m(-n)}} [v_{-n}, u_{n}] \\
&= -T_{M_r} u_n v_{-n}.
\end{split}
\end{align}
Combining (\ref{tel}) and (\ref{tel2}) yields
\begin{align} \label{therest}
\begin{split}
Z
&= \sum_{m \in \Z_{>0}} \la^m \sum_{n \in [\eps]_{>0}} T_{M_{r-mn}} [u_n, v_{-n}]
- \sum_{m \in \Z_{>0}} \la^{-m} \sum_{n \in [\eps]_{<0}} T_{M_{r+mn}} [u_n, v_{-n}] \\
&= \sum_{n \in [\eps]_{<0}} T_{M_r} u_n v_{-n} + p(u, v) \sum_{n \in [\eps]_{>0}} T_{M_r} v_{-n} u_n.
\end{split}
\end{align}

Next we have
\begin{align*}
\begin{split}
(\res_z \frac{e^{2\pi i (1+\eps) z}}{e^{2\pi i z} - 1} Y[u, z]v dz)_0
&= (\res_z \frac{e^{2\pi i (1+\eps) z}}{e^{2\pi i z} - 1} e^{2\pi i \D_u z} Y(u, e^{2\pi i z} - 1)v dz)_0 \\
&= (2\pi i)^{-1} (\res_w w^{-1} (1+w)^{\D_u+\eps} Y(u, w)v dw)_0 \\
&= (2\pi i)^{-1} \sum_{j \in \Z_+} \binom{\D_u+\eps}{j} (u_{(j-1)}v)_0.
\end{split}
\end{align*}
Plugging this into the Borcherds identity with $n = -1$ and $m = 1+\eps_u = -k$ yields
\begin{align} \label{P0reduction}
Y
= \sum_{j \in \Z_+} T_{M_r} \left( u_{-j+\eps} v_{j-\eps} + p(u, v) v_{-j-1-\eps} u_{j+1+\eps} \right).
\end{align}

Now we see that
\[
-Y + Z = \left\{
\begin{array}{ll}
-T_{M_r} a_0 b_0 & \text{if $\mu = 1$}, \\
0 & \text{otherwise}. \\
\end{array}
\right.
\]
If $\mu \neq 1$ then $X = \delta = 0$ and we are done, similarly if $\mu = \la = 1$. Finally suppose $\mu = 1 \neq \la$: from Proposition \ref{a0b} we have
\[
X - Y + Z = T_{M_r} u_0 v_0 - T_{M_r} u_0 v_0 = 0
\]
as required.
\end{proof}

%%%%%%%%%%%%%%%%%%%%%%%%%%%%%%%%%%%%%%%%%%%%%%%%%%%%%%%%%%%%%%%%%%%%%%%%%%%%%%%%%%%%%%%%%%%%%%
%%%%%%%%%%%%%%%%%%%%%%%%%%%%%%%%%%%%%%%%%%%%%%%%%%%%%%%%%%%%%%%%%%%%%%%%%%%%%%%%%%%%%%%%%%%%%%
%%%%%%%%%%%%%%%%%%%%%%%%%%%%%%%%%%%%%%%%%%%%%%%%%%%%%%%%%%%%%%%%%%%%%%%%%%%%%%%%%%%%%%%%%%%%%%
%%%%%%%%%%%%%%%%%%%%%%%%%%%%%%%%%%%%%%%%%%%%%%%%%%%%%%%%%%%%%%%%%%%%%%%%%%%%%%%%%%%%%%%%%%%%%%
%%%%%%%%%%%%%%%%%%%%%%%%%%%%%%%%%%%%%%%%%%%%%%%%%%%%%%%%%%%%%%%%%%%%%%%%%%%%%%%%%%%%%%%%%%%%%%

In the two following lemmas suppose $(\mu(u), \la(u)) = (1, 1)$.
\vspace{.3cm}
\begin{prop} \label{trtodef1}
$T_M(\res_z \wp(z, q) Y[u, z]v, \tau) = 0$.
\end{prop}

\begin{proof}
For all $u \in V$ we have $(L_{[-1]}u)_0 = 2\pi i(L_{-1}u + L_0u)_0 = 0$. By Proposition \ref{a-1b} we have
\begin{align*}
0 &= T_M((L_{[-1]}u)_0 b_0 q^{L_0}) \\
&= T_M((\res_z P(z, q) Y[L_{[-1]}u, z]v dz)_0 q^{L_0}) \\
&= T_M((\res_z P(z, q) \partial_z Y[u, z]v dz)_0 q^{L_0}) \\
&= -T_M((\res_z \partial_z P(z, q) Y[u, z]v dz)_0 q^{L_0}) \\
&= -T_M((\res_z (\wp(z, q) + G_2(q)) Y[u, z]v dz)_0 q^{L_0}).
\end{align*}
By Proposition \ref{a0b}, the $G_2(q)$ term contributes nothing, so the result follows.
\end{proof}

\vspace{.3cm}
\begin{prop} \label{trtodef2}
$\left[(2\pi i)^2 q\frac{d}{dq} + \DD_u G_2(q)\right] S_M(u, \tau) = S_M(\res_z \zeta(z, q) L[z]u dz, \tau)$.
\end{prop}

\begin{proof}
We start with Equation (\ref{pro24eq}). Multiply through by $q^{-\mathfrak{c}/24}$ and substitute $u = \tilde{\om} = (2\pi i)^2 (\om - \mathfrak{c}/24 \vac)$, so that $u_0 = (2\pi i)^2 (L_0 - \mathfrak{c}/24)$. The right hand side is
\[
-(2\pi i)^2 T_M((L_0 - \mathfrak{c}/24) v_0 q^{L_0-\mathfrak{c}/24}) = -(2\pi i)^2 q \frac{d}{dq} T_M(v_0 q^{L_0-\mathfrak{c}/24}).
\]
The left hand side is
\begin{align*}
T_M((\res_z P(z, q) L[z]v dz)_0 q^{L_0-\mathfrak{c}/24})
= {} & T_M((\res_z [-\zeta(z, q) + zG_2(q) - \pi i] L[z]v dz)_0 q^{L_0-\mathfrak{c}/24}) \\
= {} & G_2(q) \DD_v T_M(v_0 q^{L_0-\mathfrak{c}/24}) \\
& -T_M((\res_z \zeta(z, q) L[z]v dz)_0 q^{L_0-\mathfrak{c}/24}),
\end{align*}
having used $\tilde{\om}_{([1])}v = \DD_v b$ and Proposition \ref{a0b}.
\end{proof}

%%%%%%%%%%%%%%%%%%%%%%%%%%%%%%%%%%%%%%%%%%%%%%%%%%%%%%%%%%%%%%%%%%%%%%%%%%%%%%%%%%%%%%%%%%%%%%
%%%%%%%%%%%%%%%%%%%%%%%%%%%%%%%%%%%%%%%%%%%%%%%%%%%%%%%%%%%%%%%%%%%%%%%%%%%%%%%%%%%%%%%%%%%%%%
%%%%%%%%%%%%%%%%%%%%%%%%%%%%%%%%%%%%%%%%%%%%%%%%%%%%%%%%%%%%%%%%%%%%%%%%%%%%%%%%%%%%%%%%%%%%%%
%%%%%%%%%%%%%%%%%%%%%%%%%%%%%%%%%%%%%%%%%%%%%%%%%%%%%%%%%%%%%%%%%%%%%%%%%%%%%%%%%%%%%%%%%%%%%%
%%%%%%%%%%%%%%%%%%%%%%%%%%%%%%%%%%%%%%%%%%%%%%%%%%%%%%%%%%%%%%%%%%%%%%%%%%%%%%%%%%%%%%%%%%%%%%

\section{Exhausting a conformal block by supertrace functions} \label{exhaust}

Let $\phi, \psi \in \C$. We shall say $\phi$ is \emph{lower} than $\psi$ (and $\psi$ is \emph{higher} than $\phi$) if the real part of $\phi$ is strictly less than that of $\psi$.

Let $S(u, \tau) \in \CC(g, h)$. In this section we show that $S$ may be written as a linear combination of supertrace functions $S_M(u, \tau)$ for $M \in P_h(g, V)$. We need the following proposition.

\vspace{.3cm}
\begin{prop} \label{induc}
Let $S \in \CC(g, h)$ with Frobenius expansion (\ref{canonform}).
\begin{itemize}
\item Let $j \in \{1, 2, \ldots b(p)\}$, then $C_{p, j, 0}((\om - \frac{\mathfrak{c}}{24} - \la_{p, j}) * u) = 0$ for all $u \in V^g$.

\item Let $j \in \{1, 2, \ldots b(p-1)\}$, then $C_{p-1, j, 0}((\om - \frac{{\mathfrak{c}}}{24} - \la_{p-1, j})^2 * u) = 0$ for all $u \in V^g$.
\end{itemize}
\end{prop}

\begin{proof}
Recall equation (\ref{ax4'}) -- the equivalent form of {\bf CB4}. Equating coefficients of $\log^p q$ shows that (\ref{ax4'}) holds with $S_{p, j}$ in place of $S$, that is
\begin{align} \label{ax4''}
(2\pi i)^2 q\frac{d}{dq} S_{p, j}(u, \tau) = -S_{p, j}(\res_z P(z, q) L[z]u dz, \tau).
\end{align}
Let us equate coefficients of $q^{\la_{p, j}}$ in (\ref{ax4''}). The left hand side gives $(2\pi i)^2 \la_{pj} C_{p, j, 0}(u)$, while the right hand side gives $C_{p, j, 0}$ applied to
\begin{align*}
2\pi i \res_z \frac{e^{2\pi i z}}{e^{2\pi i z}-1} L[z]u dz
= {} & \res_w w^{-1} (1+w)^{\D_\om} Y((2\pi i)^2\om, w)u dw \\
& - (\mathfrak{c}/24) \res_w w^{-1} (1+w)^{\D_{\vac}} Y((2\pi i)^2 \vac, w)u dw \\
= {} & (2\pi i)^2 (\om - \mathfrak{c}/24) * u.
\end{align*}
This proves the first part.

Without loss of generality let $\la_{p, j} = \la_{p-1, j}$. Equating coefficients of $\log^{p-1} q$ in (\ref{ax4'}) yields
\begin{align} \label{2topeqn}
(2\pi i)^2 q \frac{d}{dq} S_{p-1, j}(u, \tau) = -S_{p-1, j}(\res_z P(z, q) L[z]u dz, \tau) - p (2\pi i)^2 S_{p, j}(u, \tau).
\end{align}
Equating coefficients of $q^{\la_{p-1, j}}$ yields
\begin{align*}
C_{p-1, j, 0}((\om - \mathfrak{c}/24 - \la_{p-1, j}) * u) = p C_{p, j, 0}(u).
\end{align*}
This, together with the first part of the proposition, implies the second part.
\end{proof}

Let $S$ have the Frobenius expansion (\ref{canonform}). Following Section \ref{symmfunc} we have
\[
C_{p, j, 0}(u) = \sum_{N} \alpha_N T_N(u)
\]
for some constants $\alpha_N \in \C$. The sum runs over the $h$-invariant irreducible $\zhu_g(V)$-modules $N$. Proposition \ref{induc} implies that $\al_N$ is nonzero only for $N$ that satisfy $\om|_N = \la_{p, j} + \mathfrak{c}/24$. Now consider
\[
\sum_{N} \alpha_N S_{L(N)}(u, \tau) \in q^{\la_{p, j}} \C[[q^{1/|G|}]].
\]
The coefficient of $q^{\la_{p, j}}$ is nothing but $C_{p, j, 0}$. Therefore the series
\[
S'(u, \tau) = S_{p, j}(u, \tau) - \sum_N \alpha_N S_{L(N)}(u, \tau)
\]
has lowest power of $q$ whose exponent is higher than $\la_{p, j}$.

The coefficient of the lowest power of $q$ in $S'(u, \tau)$ descends to a $h$-supersymmetric function on $\zhu_g(V)$, so we may write it as a linear combination of $T_N$. The modules $N$ that occur must be different than the ones used in the first iteration because $\om$ acts on them by some constant higher than $\la_{p, j} + \mathfrak{c}/24$. We subtract the corresponding $S_{L(N)}(u, \tau)$ as before and repeat. The process terminates because there are only finitely many irreducible $\zhu_g(V)$-modules. We obtain $S_{p, j}(u, \tau)$ as a linear combination of $S_M(u, \tau)$. It follows that $S_{p, j}(u, \tau) \in \CC(g, h)$.

We may repeat the argument above, using the second part of Proposition \ref{induc}, to conclude that $S_{p-1, j}(u, \tau) \in \CC(g, h)$ also. Hence $S_{p-1, j}$ satisfies (\ref{ax4''}) in addition to (\ref{2topeqn}). Together these equations imply $p = 0$. Thus $S = \sum_j S_{p, j}$ is a linear combination of supertrace functions.

In summary we have following explicit description of conformal blocks.
\vspace{.3cm}
\begin{thm} \label{trsp}
Let $V$ be a $C_2$-cofinite VOSA with rational conformal weights and $G$ a finite group of automorphisms of $V$. Suppose $\zhu_g(V)$ is semisimple for each $g \in G$. Fix commuting $g, h \in G$. For $M$ a $h$-invariant positive energy $g$-twisted $V$-module, select $\ga : M \rightarrow M$ satisfying
\begin{itemize}
\item $\ga^{-1} u_n \ga x = h(u)_n x$ for all $u \in V$, $n \in [\eps_u]$, $x \in M$, and

\item if $M$ is of {\sf Type II}, $A = A_0[\suv]/(\suv^2=1)$ is the simple component of $\zhu_g(V)$ corresponding to $M$, and $h(\suv) = (-1)^p \suv$, then additionally $\ga$ has parity $p$.
\end{itemize}
Then a basis of the space $\CC(g, h)$ of conformal blocks is the set of functions
\begin{align*}
S_{M}(u, \tau) = \left\{ \begin{array}{ll}
\str_{M} \left( u_0 \ga \sigma_{M}^{p(\ga)} q^{L_0 - \mathfrak{c}/24} \right) & \text{if $M$ is of {\sf Type I}}, \\
\tr_{M} \left( u_0 \ga \sigma_{M}^{p(\ga)} \xi q^{L_0 - \mathfrak{c}/24} \right) & \text{if $M$ is of {\sf Type II}}, \\
\end{array}\right.
\end{align*}
as $M$ runs over the set of $h$-invariant positive energy $g$-twisted $V$-modules.
\end{thm}
Combining Theorem \ref{trsp} with Theorem \ref{sl2zinv} on modular invariance of conformal blocks yields Theorem \ref{mythm}.

%%%%%%%%%%%%%%%%%%%%%%%%%%%%%%%%%%%%%%%%%%%%%%%%%%%%%%%%%%%%%%%%%%%%%%%%%%%%%%%%%%%%%%%%%%%%%%
%%%%%%%%%%%%%%%%%%%%%%%%%%%%%%%%%%%%%%%%%%%%%%%%%%%%%%%%%%%%%%%%%%%%%%%%%%%%%%%%%%%%%%%%%%%%%%
%%%%%%%%%%%%%%%%%%%%%%%%%%%%%%%%%%%%%%%%%%%%%%%%%%%%%%%%%%%%%%%%%%%%%%%%%%%%%%%%%%%%%%%%%%%%%%
%%%%%%%%%%%%%%%%%%%%%%%%%%%%%%%%%%%%%%%%%%%%%%%%%%%%%%%%%%%%%%%%%%%%%%%%%%%%%%%%%%%%%%%%%%%%%%
%%%%%%%%%%%%%%%%%%%%%%%%%%%%%%%%%%%%%%%%%%%%%%%%%%%%%%%%%%%%%%%%%%%%%%%%%%%%%%%%%%%%%%%%%%%%%%

\section{The neutral free fermion VOSA} \label{ex1}

The VOSA $V = F(\vp)$ studied in this section is well known and goes by several names, we call it the neutral free fermion VOSA. It is defined (\cite{Kac}, pg. 98) to be the vector superspace spanned by the monomials
\[
\vp_{n_1} \cdots \vp_{n_s} \vac,
\]
where $n_i \in 1/2 + \Z$, $n_1 < \ldots < n_s < 0$, and the monomial has parity $s \bmod 2$. The VOSA structure is generated by the single odd field
\[
Y(\vp, z) = \sum_{n \in \Z} \vp_{(n)} z^{-n-1} = \sum_{n \in 1/2 + \Z} \vp_n z^{-n-1/2}.
\]
The action of the modes on $V$ is by left multiplication, subject to the relations $\vp_n \vac = 0$ for $n > 0$ and the commutation relation
\[
\vp_m \vp_n + \vp_n \vp_m = \delta_{m, -n} \quad \Longleftrightarrow \quad [Y(\vp, z), Y(\vp, w)] = \delta(z, w).
\]
The Virasoro vector is
\[
\om = \frac{1}{2} \vp_{-3/2} \vp_{-1/2} \vac.
\]
The element $\vp = \vp_{-1/2} \vac$ has conformal weight $1/2$, and the central charge of $V$ is $\mathfrak{c} = 1/2$. This VOSA is $C_2$-cofinite.

%%%%%%%%%%%%%%%%%%%%%%%%%%%%%%%%%%%%%%%%%%%%%%%%%%%%%%%%%%%%%%%%%%%%%%%%%%%%%%%%%%%%%%%%

Let $G = \{1, \sigma_V\} \cong \Z/2\Z$. In this section we explicitly compute the conformal blocks evaluated on the vector $u = \vac$ and $u = \vp$. To do so we determine the $g$-twisted Zhu algebras (see Section \ref{coeffs}) and their modules, and then write down the supertace functions of Theorem \ref{trsp}.

Let $g = \sigma_V$. We have $\eps_\vp = -1/2$ and
\[
\vp \circ_n v = \res_w w^n (1 + w)^{1/2-1/2} Y(\vp, w)v dw = \vp_{(n)}v
\]
(Definitions \ref{twisteddefn} and \ref{zhualgdefn}). Therefore $J_{\sigma_V}$ contains all monomials except $\vac$, hence $\zhu_{\sigma_V}(V)$ is either $\C \vac$ or $0$. Since $V$ itself is a positive energy $\sigma_V$-twisted $V$-module and $\V_0 = \C \vac$ we have $\zhu_{\sigma_V}(V) = \C \vac$. The unique irreducible $\zhu_{\sigma_V}(V)$-module is $N = \C$ and the corresponding $\sigma_V$-twisted $V$-module is $L(N) = V$.

Let $h = 1$. Then we can take $\ga = 1$. Hence $\CC(\sigma_V, 1; u)$ is spanned by
\[
\str_V u_0 q^{L_0 - \mathfrak{c}/24}.
\]
Let $h = \sigma_V$. Although this automorphism acts on $\zhu_{\sigma_V}(V)$ as the identity and so $\ga|_N = N$, the extension of $\ga$ to $V$ according to equation (\ref{intertwine}) is $\sigma_V$. Hence $\CC(\sigma_V, \sigma_V; u)$ is spanned by
\[
\str_V u_0 \sigma_V q^{L_0 - \mathfrak{c}/24} = \tr_V u_0 q^{L_0 - \mathfrak{c}/24}.
\]

Let $g = 1$. We have
\[
\vp \circ_n v = \res_w w^n (1+w)^{1/2} Y(\vp, w)v dw,
\]
so in $\zhu_1(V)$ any mode $\vp_{(n)}$ for $n \leq -2$ is a linear combination of modes $\vp_{(k)}$ for $k > n$. Hence $\zhu_1(V)$ is a quotient of $\C \vac + \C \vp$. In fact $\zhu_1(V) = \C \vac + \C \vp$ which we prove below by exhibiting an irreducible positive energy $1$-twisted $V$-module $M$ such that $M_0$ is 2 dimensional. The unit element of $\zhu_1(V)$ is $\vac$ and we readily compute that $\vp * \vp = \tfrac{1}{2} \vac$. Therefore $\zhu_1(V) \cong \C[\suv] / (\suv^2 = 1)$, where $1$ is the image of $\vac$ and $\suv$ is the image of $\sqrt{2} \vp$.

We construct $M$ as follows: it has basis
\[
\vp^M_{n_1} \cdots \vp^M_{n_s} 1
\]
where $n_i \in \Z$, $n_1 < \ldots < n_s \leq 0$ and the parity of this monomial is $s \bmod 2$. The modes of the field
\[
Y^M(\vp, z) = \sum_{n \in \Z} \vp^M_n z^{-n-1}
\]
satisfy $\vp^M_n \vac = 0$ for $n \geq 1$ and $\vp^M_m \vp^M_n + \vp^M_n \vp^M_m = \delta_{m, -n}$. Note that $M_0 = \C 1 + \C \vp^M_0 1$. The unique irreducible $\zhu_{1}(V)$-module is $N = M_0$ and the corresponding $1$-twisted $V$-module is $L(N) = M$.

Let $h = 1$. According to our prescription $\ga = 1$. Hence $\CC(1, 1; u)$ is spanned by
\[
\tr_{M} u^M_0 \suv q^{L_0 - \mathfrak{c}/24}.
\]
Explicitly $\suv : M \rightarrow M$ is
\[
\suv : \vp^M_{n_1} \cdots \vp^M_{n_s} 1 \mapsto \sqrt{2} \vp^M_{n_1} \cdots \vp^M_{n_s} \vp^M_0 1.
\]

Let $h = \sigma_V$. In this case $h|_{N_0} = \text{id}_{N_0}$ but $h(\suv) = -\suv$, so $\ga_N = \sigma_N \suv$ which extends to $\ga = \sigma_M \suv$. Hence $\CC(1, \sigma_V; u)$ is spanned by
\[
\tr_M u^M_0 \ga \sigma_M \suv q^{L_0 - \mathfrak{c}/24} = -\tr_M u^M_0 q^{L_0 - \mathfrak{c}/24}.
\]

\subsection{Conformal blocks in weights $0$ and $1/2$}

We can evaluate the traces and supertraces above using some simple combinatorics. In terms of the Dedekind eta function defined by equation(\ref{etadefn}) we obtain
\begin{align*}
&\text{$\CC(\sigma_V, 1; \vac)$ is spanned by}
\,\,
q^{-1/48} \sch(V)
= q^{-1/48} \prod_{n \geq 0} (1 - q^{n+1/2})
= \frac{\eta(\tau/2)}{\eta(\tau)}, \\
&\text{$\CC(\sigma_V, \sigma_V; \vac)$ is spanned by}
\,\,
q^{-1/48} \ch(V)
= q^{-1/48} \prod_{n \geq 0} (1 + q^{n+1/2})
= \frac{\eta(\tau)^2}{\eta(2\tau)\eta(\tau/2)}, \\
&\text{$\CC(1, 1; \vac) = 0$}, \\
\text{and\,\,}&\text{$\CC(1, \sigma_V; \vac)$ is spanned by}
\,\,
q^{-1/48} q^{1/16} \ch(M)
= q^{-1/24} \prod_{n \geq 0} (1 + q^n)
= 2\frac{\eta(\tau)}{\eta(2\tau)}.
\end{align*}
The third equality is because the trace of an odd linear map vanishes, and in the fourth we have used $L_0|_{M_0} = 1/16$ which follows by direct computation with the Borcherds identity.

Theorem \ref{mythm} implies that if $f(\tau) \in \CC(g, h; \vac)$ then $f(A\tau) \in \CC((g, h) \cdot A; \vac)$ for all $A \in SL_2(\Z)$. This may be verified directly for the generators $T = \left(\begin{smallmatrix}1 & 1 \\ 0 & 1\end{smallmatrix}\right)$ and $S = \left(\begin{smallmatrix}0 & -1 \\ 1 & 0\end{smallmatrix}\right)$ using the explicit forms above together with Proposition \ref{etatrans}.

%%%%%%%%%%%%%%%%%%%%%%%%%%%%%%%%%%%%%%%%%%%%%%%%%%%%%%%%%%%%%%%%%%%%%%%%%%%%%%%%%%%%%%%%

Since $L_{[0]}\vp = L_0 \vp$, we have $\DD_\vp = \D_\vp = 1/2$. If $(g, h) \neq (1, 1)$ then $\CC(g, h; \vp)$ is spanned by the (super)trace of an odd linear map, and so is $0$. The operator $\vp^M_0 \suv$ on $M$ is explicitly (ignoring a factor of $\sqrt{2}$)
\begin{align*}
\vp^M_{n_1} \cdots \vp^M_{n_s} 1
\mapsto & \vp^M_0 \vp^M_{n_1} \cdots \vp^M_{n_s} \vp^M_0 1 = \tfrac{1}{2} \eps \vp^M_{n_1} \cdots \vp^M_{n_s} 1
\end{align*}
where $\eps = (-1)^{s-1}$ if $n_s = 0$ and $(-1)^s$ otherwise. Therefore
\begin{align*}
&\text{$\CC(1, 1; \vp)$ is spanned by}\,\,q^{-1/48} q^{1/16} \prod_{n=1}^\infty (1-q^n) = \eta(\tau), \\
&\text{$\CC(g, h; \vac) = 0$} \quad \text{for all other pairs $(g, h)$}. \\
\end{align*}

%%%%%%%%%%%%%%%%%%%%%%%%%%%%%%%%%%%%%%%%%%%%%%%%%%%%%%%%%%%%%%%%%%%%%%%%%%%%%%%%%%%%%%%%%%%%%%
%%%%%%%%%%%%%%%%%%%%%%%%%%%%%%%%%%%%%%%%%%%%%%%%%%%%%%%%%%%%%%%%%%%%%%%%%%%%%%%%%%%%%%%%%%%%%%
%%%%%%%%%%%%%%%%%%%%%%%%%%%%%%%%%%%%%%%%%%%%%%%%%%%%%%%%%%%%%%%%%%%%%%%%%%%%%%%%%%%%%%%%%%%%%%
%%%%%%%%%%%%%%%%%%%%%%%%%%%%%%%%%%%%%%%%%%%%%%%%%%%%%%%%%%%%%%%%%%%%%%%%%%%%%%%%%%%%%%%%%%%%%%
%%%%%%%%%%%%%%%%%%%%%%%%%%%%%%%%%%%%%%%%%%%%%%%%%%%%%%%%%%%%%%%%%%%%%%%%%%%%%%%%%%%%%%%%%%%%%%

\section{The charged free fermions VOSA with real conformal weights} \label{ex2}

As a vector superspace, the charged free fermions VOSA $V = F_{\text{ch}}^a(\psi, \psi^*)$ (\cite{Kac}, pg. 98) is the span of the monomials
\begin{align} \label{monomial}
\psi_{(i_1)} \psi_{(i_2)} \cdots \psi_{(i_m)} \psi^*_{(j_1)} \psi^*_{(j_2)} \cdots \psi^*_{(j_n)} \vac,
\end{align}
where $i_r, j_s \in \Z$, $i_1 < \ldots < i_m \leq -1$, $j_1 < \ldots < j_n \leq -1$, and the parity of the monomial is $(m + n) \bmod 2$. The VOSA structure is generated by the two odd fields
\[
Y(\psi, z) = \sum_{n \in \Z} \psi_{(n)} z^{-n-1} \quad \text{and} \quad Y(\psi^*, z) = \sum_{n \in \Z} \psi^*_{(n)} z^{-n-1}.
\]
The action of the modes on $V$ is by left multiplication, subject to the relations $\psi_{(n)} \vac = \psi^*_{(n)} \vac = 0$ for $n \geq 0$, and the commutation relations
\begin{align} \label{commrel}
[\psi_{(m)}, \psi^*_{(n)}] = \psi_{(m)} \psi^*_{(n)} + \psi^*_{(n)} \psi_{(m)} = \delta_{m+n, -1} \Longleftrightarrow [Y(\psi, z), Y(\psi^*, w)] = \delta(z, w).
\end{align}
All other commutators vanish.

Let $0 < a < 1$ be a real number. We define a Virasoro vector (\cite{Kac}, pg. 102)
\begin{align} \label{virdefn}
\om^a = a \left( \psi_{(-2)} \psi^*_{(-1)} \vac \right) + (1-a) \left( \psi^*_{(-2)} \psi_{(-1)} \vac \right)
\end{align}
and write $L^a(z) = Y(\om^a, z)$. With respect to this choice of Virasoro vector we have
\[
\D_\psi = 1-a \quad \text{and} \quad \D_{\psi^*} = a.
\]
The central charge of $V$ is $\mathfrak{c} = -2(6a^2-6a+1)$. This VOSA is $C_2$-cofinite. Note that $\om^a$ may be defined for any $a \in \R$ and is a Virasoro vector. However if $a$ lies outside the interval $(0, 1)$ then the conformal weights of $V$ might be unbounded below and its graded pieces might be infinite dimensional. We wish to exclude these cases.

%%%%%%%%%%%%%%%%%%%%%%%%%%%%%%%%%%%%%%%%%%%%%%%%%%%%%%%%%%%%%%%%%%%%%%%%%%%%%%%%%%%%%%%%

\subsection{Twisted modules}

Fix $\mu, \la \in \C$ of unit modulus and let the automorphisms $g$ and $h$ of $V$ be defined by
\begin{align*}
g(\psi) = \mu^{-1} \psi, \quad g(\psi^*) = \mu \psi^*, \\
h(\psi) = \la^{-1} \psi, \quad h(\psi^*) = \la \psi^*,
\end{align*}
extended to all $n^\text{th}$ products (note that the Virasoro element is indeed fixed by $g$ and $h$). Let $\mu = e^{2\pi i \delta}$ and $\la = e^{2\pi i \rho}$ where $\delta, \rho \in [0, 1)$, we do not require these to be rational numbers. Hence $g$ and $h$ need not have finite order.

A $g$-twisted $V$-module $M$ will have fields
\begin{align*}
Y^M(\psi, z) = \sum_{n \in -[\delta]} \psi^M_{n} z^{-n-(1-a)} \quad
\text{and} \quad Y^M(\psi^*, z) = \sum_{n \in [\delta]} \psi^{*M}_{n} z^{-n-a},
\end{align*}
and the modes must satisfy
\begin{align} \label{psicommrels}
[\psi^M_{m}, \psi^{*M}_{n}] = \psi^M_{m} \psi^{*M}_{n} + \psi^{*M}_{n} \psi^M_{m} = \delta_{m+n, 0}.
\end{align}

Let us put $M = V$, $\psi^M(z) = z^{-x} \psi(z)$ and $\psi^{*M}(z) = z^{x} \psi^*(z)$ where $x \in -[\delta]-[a]$. We may easily confirm that
\begin{align*}
\psi^M_{n} = \psi_{(n-a-x)} \quad \text{and} \quad \psi^{*M}_{n} = \psi^*_{(n-1+a+x)}
\end{align*}
have the correct commutation relations (\ref{psicommrels}). Therefore $M = M^{(g)} = V$ is given the structure of a $g$-twisted positive energy $V$-module (clearly irreducible too), as long as we choose $x$ so that $\psi^M_n$ and $\psi^{*M}_n$ annihilate $\vac$ for $n > 0$. This leads to the requirement $0 \leq a+x \leq 1$.

For $g \neq 1$ this requirement fixes $x$ uniquely. If $g = 1$ then $\delta = 0$ and we have the choice of putting $x = 1-a$ or $x = -a$, but both lead to the same module up to isomorphism. Indeed if $x  =1-a$ then $\psi^M_{n} = \psi_{(n-1)}$ and $\psi^{*M}_{n} = \psi^*_{(n)}$, so $M_0 = \C \vac + \C \psi$. If $x = a$ then $M_0 = \C \vac + \C \psi^*$. Then $f : M_0 \rightarrow M_0$ defined by $f(\vac) = \psi^*$ and $f(\psi) = \vac$ lifts to an equivalence from the first to the second module. It is convenient to use the $x = 1-a$ case, since then in all cases we have the uniform formula $a + x = 1 - \delta$.

%%%%%%%%%%%%%%%%%%%%%%%%%%%%%%%%%%%%%%%%%%%%%%%%%%%%%%%%%%%%%%%%%%%%%%%%%%%%%%%%%%%%%%%%

\subsection{Zhu algebras}

Let $g \neq 1$. If $\eps_v + \eps_\psi = -1$ then
\[
\psi \circ_n v = \res_w w^n (1+w)^{\D_\psi + \eps_\psi} Y(\psi, z)v dw \in J_g.
\]
It is possible to write $\psi_{(n)}v$ as a linear combination of $\psi_{(k)}v$ for $k > n$ whenever $n \leq -1$. The same goes for $\psi^*$. Iterating this procedure reveals that $\zhu_g(V)$ is a quotient of $\C \vac$. The existence of the positive energy $g$-twisted $V$-module exhibited above shows that $\zhu_g(V) = \C \vac$.

Let $g = 1$. The same argument as above holds but with $n \leq -2$. Thus $\zhu_1(V)$ is a quotient of $\C \vac + \C \psi + \C \psi^* + \C \psi_{(-1)}\psi^*$. Calculating the products of these four elements reveals that their span is isomorphic to $\en(\C^{1|1})$ via $\psi \mapsto E_{21}$, $\psi^* \mapsto E_{12}$ and $\psi_{(-1)}\psi^* \mapsto (a-1)E_{11} + aE_{22}$. Since we have constructed an irreducible positive energy $1$-twisted $V$-module we have $\zhu_1(V) \cong \en(\C^{1|1})$.

Note that Theorem \ref{mythm} has nothing to say unless $g, h$ have finite order and $a \in \R$ is chosen to lie in $\Q$. Even so we can write down supertrace functions in general. The Zhu algebras are all of {\sf Type I}, so there is a single supertrace function associated to $(g, h)$ and it is
\[
\str_M u^M_0 \ga q^{L_0 - \mathfrak{c}/24},
\]
where $M$ is the unique irreducible $g$-twisted positive energy $V$-module. Here $\ga = h^{-1}$, which we see from equation (\ref{intertwine}) plus the fact that $h$ restricts to the identity on $\zhu_g(V)$. 
%From the construction of $M^{(1)}$ we see that $\ga = h^{-1}$ in the $g = 1$ case too.

Now we restrict attention to $u = \vac$ and write
\[
\chi_{\mu, \la}(\tau) := \str_M h^{-1} q^{L_0 - \mathfrak{c}/24}
\]
for brevity. We shall express the twisted supercharacters $\chi_{\mu, \la}(\tau)$ in terms of Jacobi theta functions and derive modular transformations.

%%%%%%%%%%%%%%%%%%%%%%%%%%%%%%%%%%%%%%%%%%%%%%%%%%%%%%%%%%%%%%%%%%%%%%%%%%%%%%%%%%%%%%%%

\subsection{Twisted Supercharacters}

%%%%%%%%%%%%%%%%%%%%%%%%%%%%%%%%%%%%%%%%%%%%%%%%%%%%%%%%%%%%%%%%%%%%%%%%%%%%%%%%%%%%%%%%

Recall the \emph{conformal weight} $h = h(M)$ of an irreducible $V$-module $M$, defined to be the eigenvalue of $L^M_0$ on the lowest graded piece of $M$. We use the Borcherds identity to compute
\begin{align} \label{confweight}
h(M) = \tfrac{1}{2} (\delta-a) (\delta+a-1)
=  \tfrac{1}{2} \left[ x(x-1) + 2ax \right]
\end{align}
for the twisted $V$-modules $M$ described above.

Indeed, put $u = \psi$, $v = \psi^*$ (so that $[\eps_u ]= -[\delta]$ and $[\eps_v] = [\delta]$) in the Borcherds identity (\ref{borcherds}). Let $m = -\delta$, $k = \delta$ and denote by $\LHS(n)$ the left hand side of (\ref{borcherds}) with these choices of $u$, $v$, $m$, and $k$. We have
\begin{align*}
\LHS(-1) &= (\psi_{(-1)}\psi^*)^M_0 - (\delta+a) (\psi_{(0)}\psi^*)^M_0 \\
\text{and} \quad \LHS(-2) &= (\psi_{(-2)}\psi^*)^M_0 - (\delta+a) (\psi_{(-1)}\psi^*)^M_0 + \tfrac{1}{2}(\delta+a)(\delta+a+1) (\psi_{(0)}\psi^*)^M_0.
\end{align*}
Rearranging and using $\psi_{(0)}\psi^* = \vac$ yields
\begin{align*}
(\psi_{(-2)}\psi^*)_0 = \LHS(-2) + (\delta+a) \LHS(-1) + \tfrac{1}{2}(\delta+a)(\delta+a-1).
\end{align*}

The corresponding right hand side of (\ref{borcherds}) is
\begin{align*}
\RHS(n) = \sum_{j \in \Z_+} (-1)^j \binom{n}{j} \left[ \psi_{-\delta+n-j} \psi^*_{\delta+j-n} + (-1)^n \psi^*_{\delta-j} \psi_{-\delta+j} \right].
\end{align*}
If we apply this to $\vac$, then the first term vanishes and the second term equals $(-1)^n$. Therefore, when applied to $\vac$,
\begin{align*}
(\psi_{(-2)}\psi^*)_0 = 1 - (\delta+a) + \tfrac{1}{2}(\delta+a)(\delta+a-1).
\end{align*}

A similar calculation shows that, when applied to $\vac$,
\begin{align*}
(\psi^*_{(-2)}\psi)_0 = -(\delta+a) + \tfrac{1}{2}(\delta+a)(\delta+a+1).
\end{align*}
Combining these with (\ref{virdefn}) yields (\ref{confweight}) above (having also used $a + x = 1 - \delta$).

%%%%%%%%%%%%%%%%%%%%%%%%%%%%%%%%%%%%%%%%%%%%%%%%%%%%%%%%%%%%%%%%%%%%%%%%%%%%%%%%%%%%%%%%

Let $\phi(\tau) = \prod_{n=1}^\infty (1 - q^n)$, so that $\eta(\tau) = q^{1/24} \phi(\tau)$.

Applying $\psi_{(-n)} = \psi^M_{-n+a+x}$ to a monomial in $V$ raises its $L_0$-eigenvalue by $n-a-x = n - (1-\delta)$. Similarly $\psi^*_{(-n)} = \psi^{*M}_{-n+1-a-x}$ raises the eigenvalue by $n-\delta$. We have
\[
\chi_{\mu, \la}(\tau) = q^{h(M) - \mathfrak{c}/24} \prod_{n=1}^\infty (1 - \la q^{n-(1-\delta)}) \prod_{n=1}^\infty (1 - \la^{-1} q^{n-\delta}).
\]
The first product is the contribution of the $\psi$ terms, the second is that of the $\psi^*$ terms. Note that when $g = h = 1$ the supercharacter vanishes.

Recall the Jacobi triple product identity:
\[
\prod_{m=1}^\infty (1-z^{2m})(1 + z^{2m-1} y^2)(1 + z^{2m-1} y^{-2}) = \sum_{n \in \Z} z^{n^2} y^{2n}.
\]
Set $z = q^{1/2}$, and $y^2 = -\la q^{\delta - 1/2}$. We obtain
\begin{align*}
\chi_{\mu, \la}(\tau)
&= \frac{q^{h(M) - \mathfrak{c}/24}}{\phi(q)} \sum_{n \in \Z} (-\la)^n q^{n^2/2 + (\delta - 1/2)n} \\
&= \frac{e^{2\pi i [h(M) - (\mathfrak{c}-1)/24] \tau}}{\eta(\tau)} \theta((\delta - \tfrac{1}{2})\tau + (\rho - \tfrac{1}{2}); \tau)
\end{align*}
where $\theta(z; \tau)$ is the Jacobi theta function defined by (\ref{Jacdefn}).

%%%%%%%%%%%%%%%%%%%%%%%%%%%%%%%%%%%%%%%%%%%%%%%%%%%%%%%%%%%%%%%%%%%%%%%%%%%%%%%%%%%%%%%%

\subsection{Modular transformations}

Let $A = \delta - \tfrac{1}{2}$ and $B = \rho - \tfrac{1}{2}$. Using Proposition \ref{Jactrans} we have
\begin{align*}
\chi_{\mu, \la}(\tau + 1)
&= \frac{e^{2\pi i [h - (\mathfrak{c}-1)/24] (\tau+1)}}{\eta(\tau+1)} \theta(A(\tau+1) + B; \tau+1) \\
&= e^{2\pi i [h - (\mathfrak{c}-1)/24]} \frac{e^{2\pi i [h - (\mathfrak{c}-1)/24] \tau}}{e^{\pi i / 12} \eta(\tau)} \theta(A\tau + (B + A + \tfrac{1}{2}); \tau) \\
&= e^{2\pi i [h - \mathfrak{c}/24]} \frac{e^{2\pi i [h - (\mathfrak{c}-1)/24] \tau}}{\eta(\tau)} \theta((\delta - \tfrac{1}{2})\tau + (\delta + \rho - \tfrac{1}{2}); \tau).
\end{align*}
This is proportional to $\chi_{\mu, \la\mu}(\tau)$ the $\la\mu$-twisted supercharacter of the irreducible $\mu$-twisted $V$-module.

Using Proposition \ref{Jactrans} again we have
\begin{align*}
\chi_{\mu, \la}(-1/\tau)
&= \frac{e^{-2\pi i [h - (\mathfrak{c}-1)/24] / \tau}}{\eta(-1/\tau)} \theta((B\tau - A) / \tau; -1/\tau) \\
&= \frac{e^{-2\pi i [h - (\mathfrak{c}-1)/24] / \tau}}{(-i\tau)^{1/2} \eta(\tau)} (-i\tau)^{1/2} e^{\pi i (B\tau-A)^2 / \tau} \theta(B\tau-A; \tau) \\
&= e^{-2\pi i AB} e^{-2\pi i [h - (\mathfrak{c}-1)/24 - A^2 / 2] / \tau} \frac{e^{2\pi i [B^2/2] \tau}}{\eta(\tau)} \theta(B\tau-A; \tau).
\end{align*}
Recall the formula for $\mathfrak{c}$ in terms of $a$, we use it to obtain
\begin{align} \label{miracle}
\begin{split}
\tfrac{1}{24}(\mathfrak{c}-1) + \tfrac{1}{2} A^2
&= \tfrac{1}{8}(-4a^2+4a-1) + \tfrac{1}{2}(\delta - \tfrac{1}{2})^2 \\
&= \tfrac{1}{2}(\delta - \tfrac{1}{2})^2 - \tfrac{1}{2}(a - \tfrac{1}{2})^2 \\
&= \tfrac{1}{2}(\delta-a)(\delta+a-1) = h(M).
\end{split}
\end{align}
Hence
\begin{align*}
\chi_{\mu, \la}(-1/\tau)
&= e^{-2\pi i AB} \frac{e^{2\pi i [B^2/2] \tau}}{\eta(\tau)} \theta(B\tau-A; \tau).
\end{align*}
But by calculation (\ref{miracle}) again, $\tfrac{1}{2}B^2 = h(M') - \tfrac{1}{24}(\mathfrak{c} - 1)$ where $M'$ is the irreducible $\la$-twisted $V$-module. Therefore $\chi_{\mu, \la}(-1/\tau)$ is proportional to $\chi_{\la, \mu^{-1}}(\tau)$ the $\mu^{-1}$-twisted supercharacter of the irreducible positive energy $\la$-twisted $V$-module.

In summary: $\chi_{1, 1} = 0$ and for any $(\la, \mu) \neq (1, 1)$ and $A \in SL_2(\Z)$ we have
\begin{align*}
\chi_{\mu, \la}(A\tau) \propto \chi_{(\mu, \la)\cdot A}(\tau).
\end{align*}
If $a \in \Q$ our VOSA has rational conformal weights, and if $\mu$ and $\la$ are roots of unity then $g, h$ lie in some finite group of automorphisms of $V$. In this case the modular transformations we have just derived follow from Theorem \ref{mythm}. The computation holds in general though.

%%%%%%%%%%%%%%%%%%%%%%%%%%%%%%%%%%%%%%%%%%%%%%%%%%%%%%%%%%%%%%%%
%%%%%%%%%%%%%%%%%%%%%%%%%%%%%%%%%%%%%%%%%%%%%%%%%%%%%%%%%%%%%%%%
%%%%%%%%%%%%%%%%%%%%%%%%%%%%%%%%%%%%%%%%%%%%%%%%%%%%%%%%%%%%%%%%
%%%%%%%%%%%%%%%%%%%%%%%%%%%%%%%%%%%%%%%%%%%%%%%%%%%%%%%%%%%%%%%%
%%%%%%%%%%%%%%%%%%%%%%%%%%%%%%%%%%%%%%%%%%%%%%%%%%%%%%%%%%%%%%%%
%%%%%%%%%%%%%%%%%%%%%%%%%%%%%%%%%%%%%%%%%%%%%%%%%%%%%%%%%%%%%%%%

\section{VOSAs associated to integral lattices} \label{ex3}

Let $(Q, \left<\cdot, \cdot\right>)$ be a rank $r$ integral lattice with positive definite bilinear form $\left<\cdot, \cdot\right>$. Let $\h = Q \otimes_\Z \C$ with the induced positive definite bilinear form $\left<\cdot, \cdot\right>$. The \emph{loop algebra} $\tilde{\h} = \h[t^{\pm 1}]$ is equipped with a Lie bracket as follows:
\[
[ht^m, h't^n] = m \left<h, h'\right> \delta_{m, -n}.
\]
Let $S_-(\h) = U(\tilde{\h}) / U(\tilde{\h}) \h[t]$. We write $h_m$ for $h t^m$. Explicitly $S_-(\h)$ has a basis of monomials
\[
h^1_{n_1} \cdots h^s_{n_s} 1
\]
where the $h^i$ range over a basis of $\h$, and $n_1 \leq \ldots n_s \leq -1$ are integers.

The \emph{twisted group algebra} $\C_\eps[Q]$ of $Q$ is a unital associative algebra with basis $\{e^\alpha | \al \in Q\}$, unit element $1 = e^0$, and multiplication $e^\alpha e^\beta = \eps(\alpha, \beta) e^{\alpha + \beta}$ where the function $\eps : Q \times Q \rightarrow \{\pm 1\}$ has been chosen to satisfy
\begin{itemize}
\item $\eps(0, a) = \eps(a, 0) = 1$ for all $a \in Q$,

\item $\eps(a, b)\eps(a+b, c) = \eps(a, b+c)\eps(b, c)$ for all $a, b, c \in Q$,

\item $\eps(a, b) = \eps(b, a) (-1)^{\left<a, b\right> + \left<a, a\right>\left<b, b\right>}$ for all $a, b \in Q$.
\end{itemize}
It may be shown that such $\eps$ exists.

Associated to this data there is a VOSA (\cite{Kac}, pg. 148).
\begin{defn}
The lattice VOSA $(V_Q, \vac, Y, \om)$ associated to $Q$ is defined to be $V_Q = S_-(\h) \otimes \C_\eps[Q]$ as a vector superspace, where the parity of $s \otimes e^\al$ is $\left<\al, \al\right> \bmod{2}$. The vacuum vector is $\vac = 1 \otimes 1$. Let $h \in \h$, $\al \in Q$, and $n \in \Z$. Define $h_n : V_Q \rightarrow V_Q$ by
\begin{align*}
h_n(s \otimes e^\al) &= (h_n s) \otimes e^\al \quad \text{for $n < 0$}, \\
h_n(1 \otimes \C_\eps[Q]) &= 0 \quad \text{for $n > 0$}, \\
h_0(1 \otimes e^\al) &= \left<h, \al\right> 1 \otimes e^\al, \\
\text{and} \quad [h_m, h'_n] &= m \left<h, h'\right> \delta_{m, -n}.
\end{align*}
Put $h(z) = \sum_{n \in \Z} h_n z^{-n-1}$ and
\[
\Gamma_\al(z) = e^\al z^{\al_0} \exp\left[-\sum_{j < 0} \frac{z^{-j}}{j} \al_j\right] \exp\left[-\sum_{j > 0} \frac{z^{-j}}{j} \al_j\right],
\]
where by definition $e^\al(s \otimes e^\beta) = \eps(\al, \beta) (s \otimes e^{\al + \beta})$. The state-field correspondence is given by $Y(h \otimes 1, z) = h(z)$, $Y(1 \otimes e^\al, z) = \Gamma_\al(z)$, and is extended to all of $V_Q$ by normally ordered products, i.e.,
\begin{align*}
Y(h_na, w) = \res_z \left[ h(z) Y(u, w) i_{z, w}(z-w)^n - Y(u, w) h(z) i_{w, z}(z-w)^n \right] dz.
\end{align*}
Finally, the Virasoro vector is
\[
\om = \frac{1}{2} \sum_{i=1}^r a^i_{(-1)} b^i_{(-1)} \vac
\]
where $\{a^i\}$ and $\{b^i\}$ are bases of $\h$ dual under $\left<\cdot, \cdot\right>$, i.e., $\left<a^i, b^j\right> = \delta_{ij}$.
\end{defn}

Some commutators between the generating fields are
\begin{align*}
[h(z), h'(w)] &= \left<h, h'\right> \partial_w \delta(z, w), \\
[h(z), \G_\al(w)] &= \left<\al, h\right> \G_\al(w) \delta(z, w).
\end{align*}
There is an explicit expression for $[\G_\al(z), \G_\beta(w)]$ which we will not require. The conformal weight of $h = h_{-1} \vac \in S_-(\h) \otimes 1$ is $1$, the conformal weight of $1 \otimes e^\al$ is $\left<\al, \al\right> / 2$. The central charge of $V$ equals the rank $r$ of $Q$. The lattice VOSAs are known to be $C_2$-cofinite.

%%%%%%%%%%%%%%%%%%%%%%%%%%%%%%%%%%%%%%%%%%%%%%%%%%%%%%%%%%%%%%%%%%%%%%%%%%%%%%%%%%%%%%%%

\subsection{Irreducible modules and their (super)characters}

Let $G = \{1, \sigma_V\}$. It is explained in \cite{Kac} that $V_Q$ is $\sigma_V$-rational. If $Q$ is an even lattice, i.e., $\left<\al, \al\right> \in 2\Z$ for all $\al \in Q$, then $V_Q$ is purely even and so $\sigma_V = 1$. We focus on the case $Q$ is not even. The same proof as in \cite{Kac} shows that $V_Q$ is also $1$-rational.

Let $Q^\circ \subseteq \h$ be the lattice dual to $Q$ and let $\delta \in Q^\circ$. We define
\begin{align} \label{latticetwist}
\begin{split}
Y^\delta(h, z) &= h(z) + \left<\delta, h\right> z^{-1}, \\
Y^\delta(e^\al, z) &= z^{\left<\delta, \al\right>} \G_\al(z).
\end{split}
\end{align}
Under this modified state-field correspondence $V$ acquires the structure of a positive energy $\sigma_V$-twisted $V$-module (the fields $Y^\delta(u, z)$ involve only integer powers of $z$) which depends on $\delta$ only through $\delta + Q$. Via this construction the cosets of $Q^\circ$ modulo $Q$ are in bijection with the irreducible positive energy ($\sigma_V$-twisted) $V_Q$-modules \cite{Kac}. The Virasoro field acts on $(V, Y^\delta)$ as
\begin{align*}
L^\delta(z) = L(z) + \frac{1}{2} z^{-1} \sum_{i=1}^r \left[\left<\delta, b^i\right> a^i(z) + \left<\delta, a^i\right> b^i(z) \right] + \frac{\left<\delta, \delta\right>}{2} z^{-2}.
\end{align*}
From this we see that the $L^\delta_0$-eigenvalue of $1 \otimes e^\al$ is $\left<\al+\delta, \al+\delta\right>/2$. The $L^\delta_0$-eigenvalue of $h \in \h$ is $1$ as before.

In a similar way the irreducible positive energy $1$-twisted $V$-modules are exactly $(V, Y^\rho)$ (defined as in (\ref{latticetwist})) but for $\rho \in \h$ satisfying
\begin{align} \label{Ramweight1}
\left<2\rho, \al\right> \equiv \left<\al, \al\right> \pmod 2
\end{align}
for every $\al \in Q$.

Let $\{a^i\}$ be a basis of $Q$ and let $p(a^i)$ denote the parity of $\left<a^i, a^i\right>$. Let $\{b^i\}$ be the basis of $Q^\circ$ dual to $\{a^i\}$, and let
\[
\rho = \frac{1}{2} \sum_{p(a^i)=1} b^i + \sum_{p(a^i) = 0} b^i.
\]
Clearly $\rho$ satisfies equation (\ref{Ramweight1}) for $\al \in \{a^i\}$. Now let $\al = \sum_i k_i a^i$ where $k_i \in \Z$. Then
\begin{align*}
\left<2\rho, \al\right> = \sum_{i} k_i \left<2\rho, a^i\right> \quad \text{and} \quad \left<\al, \al\right> = \sum_{i, j} k_i k_j \left<a^i, a^j\right>
\end{align*}
but
\begin{align*}
\sum_{i, j} k_i k_j \left<a^i, a^j\right> \equiv \sum_i k_i^2 \left<a^i, a^i\right> \equiv \sum_i k_i \left<a^i, a^i\right> \equiv \sum_{i} k_i \left<2\rho, a^i\right> \pmod 2,
\end{align*}
so $\rho$ satisfies (\ref{Ramweight1}) for all $\al \in Q$. Let $Q^\bullet$ be the set of all elements of $\h$ satisfying (\ref{Ramweight1}) for all $\al \in Q$.
%If $Q$ is an even lattice then $V_Q$ is a VOA rather than a VOSA, $\sigma_V = 1$, and $Q^\bullet = Q^\circ$.
 As we have taken $Q$ not even, $Q^\bullet \neq Q^\circ$ and $Q^\circ \cup Q^\bullet$ is a lattice containing $Q^\circ$ as an index $2$ sublattice.

All simple components of $\zhu_g(V)$ (and hence all irreducible $V$-modules) are $h$-invariant for each $h \in G$. For $h = 1$ this is obvious, and for $h = \sigma_V$ it follows because each simple component has a unit element which is even and hence fixed by $\sigma_V$. Let us consider $\CC(g, h; \vac)$ only, so that we may ignore any contribution of modules of {\sf Type II}. Then if $h = 1$ (resp. $\sigma_V$) we have $\ga = 1$ (resp. $\sigma_M$). The space $\CC(g, h; \vac)$ is spanned by $\str_M h q^{L^\delta_0 - \mathfrak{c}/24}$ where $M$ ranges over the set of irreducible $g$-twisted positive energy $V$-modules.

The bosonic part $S_-(\h)$ of the tensor product $S_-(\h) \otimes \C_\eps[Q]$ is purely even and $q^{-\mathfrak{c}/24} \tr_{S_-(\h)} q^{L_0} = \eta(\tau)^{-r}$. The contribution of $\C_\eps[Q]$ is given in terms of classical lattice theta functions:
\begin{align*}
\Theta^{\text{even}}_{\delta, Q}(q) &= \tr_{\C_\eps[Q]} q^{L^\delta_0} = \sum_{\al \in Q} e^{\pi i \tau \left<\al+\delta, \al+\delta\right>} \\
\quad \text{and} \quad
\Theta^{\text{odd}}_{\delta, Q}(q) &= \str_{\C_\eps[Q]} q^{L^\delta_0} = \sum_{\al \in Q} e^{\pi i \tau \left<\al+\delta, \al+\delta\right>} e^{\pi i \left<\al, \al\right>}
\end{align*}
We see that $\CC(g, h; \vac)$ is spanned by
\begin{align*}
\begin{array}{lll}
\Theta^{\text{even}}_{\delta, Q}(q) / \eta(\tau)^r  & \text{for $\delta \in Q^\circ / Q$} & \text{if $(g, h) = (\sigma_V, \sigma_V)$}, \\
\Theta^{\text{even}}_{\delta, Q}(q) / \eta(\tau)^r  & \text{for $\delta \in Q^\bullet / Q$} & \text{if $(g, h) = (1, \sigma_V)$}, \\
\Theta^{\text{odd}}_{\delta, Q}(q) / \eta(\tau)^r  & \text{for $\delta \in Q^\circ / Q$} & \text{if $(g, h) = (\sigma_V, 1)$}, \\
\Theta^{\text{odd}}_{\delta, Q}(q) / \eta(\tau)^r  & \text{for $\delta \in Q^\bullet / Q$} & \text{if $(g, h) = (1, 1)$}. \\
\end{array}
\end{align*}
The transformation
\[
\sum_{\al \in Q} e^{-i \pi \left<\al, \al\right> / \tau} = (\disc Q)^{-1/2} (-i\tau)^{r/2} \sum_{\beta \in Q^\circ} e^{i\pi \tau \left<\beta, \beta\right>}
\]
of the usual lattice theta function under $\tau \mapsto -1/\tau$ is proved using Poisson summation \cite{Igusa} (here $\disc Q$ is the discriminant, defined to be the determinant of the Gram matrix of an integral basis of $Q$). In the same way the $SL_2(\Z)$ transformations of $\Theta^{\text{even}}$ and $\Theta^{\text{odd}}$ may be deduced and it is confirmed that
\[
\slmat : \CC(g, h; \vac) \rightarrow \CC(g^\text{\sf a}h^\text{\sf c}, g^\text{\sf b}h^\text{\sf d}; \vac).
\]


\clearpage

\begin{thebibliography}{99}

\bibitem{AM}
Adamovi\'{c}, D., Milas, A.:
\newblock Vertex operator algebras associated to modular invariant representations for $A_1^{(1)}$.
\newblock Math. Res. Lett. \textbf{2}, 563-575 (1995)

\bibitem{Arike}
Arike, Y.:
\newblock Some remarks on symmetric linear functions and pseudotrace maps.
\newblock Proc. Japan Acad. Ser. A Math. Sci. \textbf{86}, 119-124 (2010)

\bibitem{Bor}
Borcherds, R.:
\newblock Vertex algebras, Kac-Moody algebras, and the Monster.
\newblock Proc. Nat. Acad. Sci. U.S.A. \textbf{83}, 3068-3071 (1986)

\bibitem{Borlater}
Borcherds, R.:
\newblock Monstrous moonshine and monstrous Lie superalgebras.
\newblock Invent. Math. \textbf{109}, 405-444 (1992)

\bibitem{DK}
De Sole, A., Kac, V.:
\newblock Finite vs affine $W$-algebras.
\newblock Jpn. J. Math. \textbf{1}, 137-261 (2006)

\bibitem{DLMorbifold}
Dong, C., Li, H., Mason, G.:
\newblock Modular-invariance of trace functions in orbifold theory and generalized Moonshine.
\newblock Comm. Math. Phys. \textbf{214}, 1-56 (2000)

\bibitem{DLMzhualgebra}
Dong, C., Li, H., Mason, G.:
\newblock Twisted representations of vertex operator algebras.
\newblock Math. Ann. \textbf{310}, 571-600 (1998)

\bibitem{DLMadmiss}
Dong, C., Li, H., Mason, G.:
\newblock Vertex operator algebras associated to admissible representations of $\hat{\mathfrak{sl}}_2$.
\newblock Comm. Math. Phys. \textbf{184}, 65-93 (1997)

\bibitem{DZ}
Dong, C., Zhao, Z.:
\newblock Modularity in Orbifold Theory for Vertex Operator Superalgebras.
\newblock Comm. Math. Phys. \textbf{260}, 227-256 (2005)

\bibitem{DZ2}
Dong, C., Zhao, Z.:
\newblock Modularity of trace functions in orbifold theory for $Z$-graded vertex operator superalgebras. In: \emph{Moonshine: the first quarter century and beyond}.
\newblock London Math. Soc. Lecture Note Ser. Vol. \textbf{372}, Cambridge: Cambridge Univ. Press, 2010, pp. 128-143

\bibitem{DZother}
Dong, C., Zhao, Z.:
\newblock Twisted representations of vertex operator superalgebras.
\newblock Commun. Contemp. Math. \textbf{8}, 101-121 (2006)

\bibitem{FLM}
Frenkel, I., Lepowsky, J., Meurman, A.:
\newblock \emph{Vertex operator algebras and the Monster},
\newblock Pure Appl. Math. Vol. \textbf{134}. Boston: Academic Press, 1988

\bibitem{Huang}
Huang, Y.:
\newblock Differential equations, duality and modular invariance.
\newblock Commun. Contemp. Math. \textbf{7}, 649-706 (2005)

\bibitem{Igusa}
Igusa, J.:
\newblock \emph{Theta Functions}.
\newblock Die Grundlehren der mathematischen Wissenschaften Band \textbf{194}. New York-Heidelberg: Springer-Verlag, 1972

\bibitem{Jordan}
Jordan, A.:
\newblock A super version of Zhu's theorem.
\newblock Ph.D. Thesis, University of Oregon, 2008.

\bibitem{KacAlgebraic}
Kac, V.:
\newblock Classification of simple algebraic supergroups.
\newblock Uspekhi Mat. Nauk. \textbf{32}, 214-215 (1977)

\bibitem{Kacsuper}
Kac, V.:
\newblock Lie superalgebras.
\newblock Adv. in Math. \textbf{26}, 8-96 (1977)

\bibitem{Kac}
Kac, V.:
\newblock \emph{Vertex algebras for beginners}.
\newblock University Lecture Series Vol. \textbf{10}, Second ed., Providence: Amer. Math. Soc., 1998

\bibitem{KacPeter}
Kac, V., Peterson, D.:
\newblock Infinite-dimensional Lie algebras, theta functions and modular forms.
\newblock Adv. in Math. \textbf{53}, 125-264 (1984)

\bibitem{KacWang}
Kac, V., Wang, W.:
\newblock Vertex operator superalgebras and their representations. In: \emph{Mathematical aspects of conformal and topological field theories and quantum groups}.
\newblock Contemp. Math. \textbf{175}, Providence: Amer. Math. Soc., 1994, pp. 161-191

\bibitem{Lang}
Lang, S.:
\newblock \emph{Elliptic Functions},
\newblock Graduate Texts in Mathematics Vol. \textbf{112}, Second ed., New York: Springer-Verlag, 1987

\bibitem{MasonTuiteZuevsky}
Mason, G., Tuite, M., Zuevsky, A.:
\newblock Torus $n$-point functions for $\R$-graded vertex operator superalgebras and continuous fermion orbifolds.
\newblock Comm. Math. Phys. \textbf{283}, 305-342 (2008)

\bibitem{M2}
Miyamoto, M.:
\newblock Intertwining operators and modular invariance.
\newblock \texttt{arXiv:math/0010180}

\bibitem{Mnonrational}
Miyamoto, M.:
\newblock Modular invariance of vertex operator algebras satisfying $C_2$-cofiniteness.
\newblock Duke Math. J. \textbf{122}, 51-91 (2004)

\bibitem{Wall}
Wall, C.:
\newblock Graded Brauer Groups.
\newblock J. Reine Angew. Math. \textbf{213}, 187-199 (1964)

\bibitem{WW}
Whittaker, E., Watson, G.:
\newblock \emph{A course of modern analysis}.
\newblock Fourth ed., Cambridge: Cambridge University Press, 1996

\bibitem{Y}
Yamauchi, H.:
\newblock Orbifold Zhu theory associated to intertwining operators.
\newblock J. Algebra \textbf{265}, 513-538 (2003)

\bibitem{Xu}
Xu, X.:
\newblock \emph{Introduction to Vertex Operator Superalgebras and their Modules}.
\newblock Math. Appl. Vol. \textbf{456}, Dordrecht: Kluwer Academic Publishers, 1998.

\bibitem{Zhu}
Zhu, Y.:
\newblock Modular invariance of characters of vertex operator algebras.
\newblock J. Amer. Math. Soc. \textbf{9}, 237-301 (1996)

\end{thebibliography}
\end{document}